\documentclass[reqno]{amsart}

\usepackage[utf8]{inputenc}
\usepackage{amsmath}
\usepackage{amsthm}
\usepackage{amsfonts}
\usepackage{amssymb}
\usepackage{mathrsfs}
\usepackage[]{hyperref}
\usepackage{cite}

\newtheorem{theorem}{Theorem}
\theoremstyle{plain}

\newtheorem{definition}{Definition}

\newtheorem{lemma}{Lemma}
\newtheorem{proposition}{Proposition}

\numberwithin{equation}{section}

\newcommand{\naturals}{\mathbb{N}}
\newcommand{\reals}{\mathbb{R}}
\newcommand{\complexes}{\mathbb{C}}
\newcommand{\integers}{\mathbb{Z}}

\newcommand{\spec}[1]{\operatorname{\sigma}{\left(#1\right)}}
\newcommand{\proj}[1]{\mathsf{P}_{#1}}
\newcommand{\tr}[1]{\operatorname{tr}{#1}}

\newcommand{\id}[1]{\mathrm{id}_{#1}}
\newcommand{\comm}[2]{[#1,#2]}
\newcommand{\acomm}[2]{\{#1,#2\}}
\newcommand{\adj}{\prime}
\newcommand{\hadj}{*}
\newcommand{\iprod}[2]{\left\langle #1, #2 \right\rangle}
\newcommand{\hsiprod}[2]{\left\langle #1, #2 \right\rangle_{2}}

\newcommand{\tnorm}[1]{\left\| #1 \right\|_{1}}
\newcommand{\hsnorm}[1]{\left\| #1 \right\|_{2}}

\newcommand{\im}[1]{\operatorname{im}{#1}}
\newcommand{\conv}[1]{\operatorname{Conv}{\left(#1\right)}}

\newcommand{\matr}[1]{\mathbb{M}_{#1}(\complexes)}

\newcommand{\cp}[2]{\operatorname{\mathsf{CP}}(#1,#2)}
\newcommand{\cpe}[1]{\operatorname{\mathsf{CP}}(#1)}
\newcommand{\cocp}[2]{\operatorname{\mathsf{coCP}}(#1,#2)}
\newcommand{\cocpe}[1]{\operatorname{\mathsf{coCP}}(#1)}
\newcommand{\dec}[2]{\operatorname{\mathsf{D}}(#1,#2)}
\newcommand{\dece}[1]{\operatorname{\mathsf{D}}(#1)}
\newcommand{\transpose}{\mathsf{T}}
\renewcommand{\transpose}{\tau}

\begin{document}
\title[Covariant decomposable maps on C*-algebras...]{Covariant decomposable maps on C*-algebras and quantum dynamics}
\author{Krzysztof Szczygielski}
\address[K. Szczygielski]
{Institute of Theoretical Physics and Astrophysics, Faculty of Mathematics, Physics and Informatics, University of Gda\'nsk, Wita Stwosza 57, 80-308 Gda\'nsk, Poland}%
\email[K. Szczygielski]{krzysztof.szczygielski@ug.edu.pl}

\begin{abstract}
We characterize covariant positive decomposable maps between unital C*-algebras in terms of a dilation theorem, which generalizes a seminal result by H.~Scutaru from Rep.~Math.~Phys. \textbf{16} (1):79--87, 1979. As a case study, we provide a certain characterization of the operator sum representation of maps on $\matr{n}$, covariant with respect to the maximal commutative subgroup of \texorpdfstring{$\mathrm{U}(n)$}{U(n)}. A connection to quantum dynamics is established by specifying sufficient and necessary conditions for covariance of D-divisible (decomposably divisible) quantum evolution families, recently introduced in J.~Phys.~A: Math.~Theor.~\textbf{56} (2023) 485202.
\end{abstract}

\maketitle

\section{Introduction}

Covariant positive linear maps gained a lot of attention due to their significance in quantum physics, where they serve as a theoretical framework for modeling quantum channels, or quantum evolutions, of systems with high level of internal symmetry. For instance, covariance appears in quite a natural manner in the widely used Davies approach to Markovian semigroups \cite{Davies1974,Davies1976} and to their generalizations in the weak coupling limit regime (see e.g.~\cite{Szczygielski2013,Szczygielski2014}). Covariant quantum channels found numerous applications in open quantum systems theory, quantum metrology, quantum optics and notably in quantum information. For example, they were employed for characterization of entanglement in spin systems, description of spin chains, modeling the evolution in presence of processes of pure dephasing, absorption and emission, description of spin-boson model and many others; see e.g.~\cite{Siudzinska2022,Siudzinska2023} and references therein for a more comprehensive summary. The algebraic framework of covariant maps was established by H.~Scutaru \cite{Scutaru1979} who characterized completely positive covariant maps from a C*-algebra into $B(H)$, $H$ being a Hilbert space, by providing an explicit construction in the spirit of Stinespring \cite{Stinespring_1955}. An important result by A.~Holevo related covariance of completely positive unital semigroups to properties of their generators \cite{Holevo1993,Holevo1996}.

The article is structured as follows. In Section \ref{sec:Preliminaries} we provide some necessary mathematical background, which includes Frobenius Hermitian basis in $\matr{n}$, transposition map on general $B(H)$ and elementary notions of various classes of positive maps on algebras. This is then followed by the core part of the article, namely Sections \ref{sec:CovariantDecMaps} and \ref{sec:Ddiv}, which contain all main results. In particular, in Section \ref{sec:CovariantDecMaps} we extend results of Scutaru \cite{Scutaru1979} onto the more general class of \emph{decomposable} positive linear maps, in the spirit of Stinespring and St{\o}rmer; this is formulated in a form of Theorems \ref{thm:MainResult1} and \ref{thm:MainResult2}. As a kind of a case study, in Section \ref{sec:MapsOnMn} we focus on a particular case of maps covariant with respect to the so-called maximal commutative subgroup of $\mathrm{U}(n)$ and provide characterization for their operator sum representation (Proposition \ref{prop:MapHPCPcoCP}, Theorems \ref{thm:PhiUconjcovariant} and \ref{thm:MaxCovDec}). The second core part, Section \ref{sec:Ddiv}, is then devoted to quantum dynamics and it mainly focuses on recently introduced \emph{D-divisible} (decomposably divisible) quantum evolution families \cite{Szczygielski2023,Siudzinska2025}. In particular, a characterization of time-local regular generators of covariant, D-divisible evolution families is provided (Proposition \ref{prop:LcovDec}, Theorems \ref{thm:SemigroupStructureTheorem} and \ref{thm:DdivisibleStructureTheorem}) which can be seen as a generalization of result by Holevo \cite{Holevo1993} applied for C*-algebras of square matrices.

\section{Preliminaries}
\label{sec:Preliminaries}

Here, we briefly outline most basic notions necessary for formulating further results, including elementary facts and constructs of a theory of positive maps on algebras.

\subsection{Notation}

Throughout the article we will use a common notation. $B(X,Y)$ (resp.~$B(X)$) will be the Banach space of bounded linear maps between normed spaces $X$ and $Y$ (resp.~on $X$). $\matr{n}$ will be the algebra of complex square matrices of size $n$. For $A\in B(H)$, $H$ a Hilbert space, we denote with $A^\hadj$ the Hermitian conjugate of $A$ and with $\tnorm{A} = \tr{\sqrt{A^\hadj A}}$ and $\hsnorm{A} = \sqrt{\tr{A^\hadj A}}$ the \emph{trace norm} and \emph{Frobenius (Hilbert-Schmidt) norm}, respectively, with the latter induced by Frobenius (Hilbert-Schmidt) inner product $\hsiprod{A}{B} = \tr{A^\hadj B}$. Space $B(H)$ with its standard operator (spectral) norm is a C*-algebra (and so is $\matr{n}$), and becomes a Banach *-algebra when either $\tnorm{\cdot}$ or $\hsnorm{\cdot}$ is taken instead. For any linear operator $A$, symbol $\spec{A}$ will denote the spectrum of $A$. For a set $S$ in an algebra $\mathscr{A}$ we denote with $S^\prime$ the commutant of $S$.

\subsection{Algebra \texorpdfstring{$\matr{n}$}{of matrices} and Frobenius basis}
\label{sec:FrobeniusBasis}

In Sections \ref{sec:MapsOnMn} and \ref{sec:DdivisibleFamilies}, devoted to finite-dimensional C*-algebras, we will make an extensive use of Frobenius-orthonormal bases spanning $\matr{n}$. In particular, two of them will be of importance: the \emph{canonical} one, $\{E_{ij}\}_{i,j=1}^{n}$ for $E_{ij}$ containing $1$ at place $(i,j)$ and zeros elsewhere, and \emph{Hermitian} one, $\{F_i\}_{i=1}^{n^2}$, consisting of matrices $F_i$ subject to
\begin{equation}\label{eq:FbasisProp}
F_i= F_{i}^{\hadj}, \quad \tr{F_i F_j} = \delta_{ij}, \quad F_{n^2} = \frac{1}{\sqrt{n}}I, \quad \tr{F_i} = \delta_{i n^2}.
\end{equation}
We choose such basis in a following, common way: matrices $F_i$ for $1 \leqslant i \leqslant \binom{n}{2}$ are chosen \emph{symmetric}, for $1+\binom{n}{2} \leqslant i \leqslant n(n-1)$ \emph{antisymmetric}, and for $1+n(n-1) \leqslant i \leqslant n^2$ \emph{diagonal}; we will sometimes denote them with $F_{i}^{\mathrm{s}}$, $F_{i}^{\mathrm{a}}$ and $F_{i}^{\mathrm{d}}$, respectively. They span mutually orthogonal subspaces in $\matr{n}$ of all symmetric, antisymmetric and diagonal matrices which we will respectively denote $\mathrm{Sym}_{n}(\complexes)$, $\mathrm{Asym}_{n}(\complexes)$ and $\mathrm{Diag}_{n}(\complexes)$. General structure of $F_i$ is the following \cite{Hioe1981,Kimura2003}:
\begin{itemize}
	\item when $F_i$ is either \emph{symmetric} or \emph{antisymmetric}:
	\begin{equation}\label{eq:FiSymAsym}
		F_{i}^{\mathrm{s}} = \frac{1}{\sqrt{2}}(E_{\mu\nu} + E_{\nu\mu}) \quad \text{or} \quad F_{i}^{\mathrm{a}} = -\frac{i}{\sqrt{2}}(E_{\mu\nu} - E_{\nu\mu})
	\end{equation}
	for some pair $(\mu,\nu)\in \{1, ... ,n\}^2$ uniquely determined by $i$, in symmetric or antisymmetric case, respectively;
	\item when $F_i$ is chosen \emph{diagonal}: first, define matrix $K_k$ as
	\begin{equation}\label{eq:FiDiag}
		K_k = \frac{1}{\sqrt{k(k+1)}}\left( \sum_{j=1}^{k}E_{jj} - kE_{k+1,k+1} \right)
	\end{equation}
	for $k \in \{1, ... , n-1\}$ and set $F_{i}^{\mathrm{d}} = K_{i-n(n-1)}$, $F_{n^2} = \frac{1}{\sqrt{n}}I$.
\end{itemize}
Such constructed basis naturally satisfies properties \eqref{eq:FbasisProp} and will be referred to simply as the \emph{Frobenius basis}. All $F_{i}^{\mathrm{s}}$, $F_{i}^{\mathrm{a}}$ are off-diagonal and have only two non-zero matrix elements, while $F_{i}^{\mathrm{d}}$ are of zero trace, except for $F_{n^2}$. Frobenius bases so constructed may be then regarded as natural generalizations of (normalized) Pauli and Gell-Mann matrices, which respectively span $\matr{2}$ and $\matr{3}$.

\subsection{Transposition on \texorpdfstring{$B(H)$}{B(H)}}
\label{sec:Transposition}

Let a complex Hilbert space $H$ be spanned by basis ${\{e_i : i\in \mathcal{I}\}}$ for some set $\mathcal{I}$ of indices (possibly uncountable). Let $J_H : H\to H$ be an antilinear isometric involutive isomorphism given via
\begin{equation}\label{eq:Jh}
	J_H x = \sum_{i\in\mathcal{I}}\overline{\iprod{e_i}{x}}e_i ,
\end{equation}
i.e.~a \emph{complex conjugation} of $x$ in basis $\{e_i\}$; naturally, one has $J_{H}^{2} x = x$. For any $A \in B(H)$, the linear operator
\begin{equation}\label{eq:TransposeDef}
	\transpose (A) = A^\transpose : H \to H, \quad \transpose (A) = J_H  A^\hadj J_H
\end{equation}
will be called the \emph{transpose of $A$} \cite{Labuschagne2006}. Every Hilbert space $H$ is isometrically anti-isomorphic to its dual $H^*$, i.e.~there exist antilinear isometry $\eta : H \to H^{*}$ given by $\eta(y) = \iprod{y}{\cdot}$. Then, $\eta$ gives rise to the \emph{Hermitian adjoint} $A^\hadj \in B(H)$ defined by $\eta(y)\circ A = \eta(A^\hadj y)$, or simply $\iprod{x}{Ay} = \iprod{A^\hadj x}{y}$, for all $x,y \in H$. Naturally, map $\eta(y)\mapsto \eta(A^\hadj y)$ is a transformation of a functional $\eta(y)$ in $H^*$. Then, a mapping $\zeta = \eta\circ J_H$, $\zeta (y) = \iprod{J_H y}{\cdot}$ can be checked to be a \emph{linear} isometry between $H$ and $H^*$, which in turn gives rise to the \emph{transpose} of $A$: indeed, taking $A : H\to H$, any $x,y \in H$ and utilizing involutiveness of $J_H$,
\begin{align}
\zeta (y)(Ax) &= \iprod{J_H y}{Ax} = \iprod{J_H J_H A^\hadj J_H y}{x} \\
&= \zeta(J_H A^\hadj J_H y)(x) \nonumber = \zeta (A^\transpose y)(x).\nonumber
\end{align}
More generally, let $H_1$, $H_2$ be Hilbert spaces endowed respectively with complex conjugations $J_1$, $J_2$ (in some chosen bases), antilinear isometries $\eta_1 , \eta_2$, and linear ones $\zeta_1$, $\zeta_2$, all defined as above. Let $A \in B(H_1, H_2)$. Then, $A^\hadj : H_{2}\to H_{1}$ satisfies $\eta_2(y)\circ A = \eta_1 (A^\hadj y)$, $y \in H_2$, while the \emph{transpose} must obey $\zeta_2(y)\circ A = \zeta_1 (A^\transpose y)$. Then, a simple computation shows that
\begin{equation}
\zeta_2(y)(Ax) = \zeta_1(J_1 A^\hadj J_2 y)(x),
\end{equation}
and so the transpose of $A : H_1 \to H_2$ is
\begin{equation}
A^\tau : H_2 \to H_1, \quad A^\tau = J_1 A^\hadj J_2 .
\end{equation}
Both $A^\hadj$ and $A^\transpose$ can be then interpreted as (representations of) liftings of operator $A$ to maps between dual spaces. For finite-dimensional spaces, $\transpose$ will simply mean the matrix transposition.

By analogy to matrix algebras, let us define the \emph{conjugate} $\overline{A}$ of operator $A : H_1 \to H_2$ by
\begin{equation}
\overline{A} : H_1 \to H_2, \quad \overline{A} = (A^\hadj)^\tau.
\end{equation}
A simple exercise is to show validity of the above lemma:

\begin{lemma}
\label{lemma:BarTransProperties}
Let $A \in B(H_1, H_2)$, $B \in B(H_2, H_3)$ for $H_1$, $H_2$, $H_3$ Hilbert spaces. The following statements hold:
\begin{enumerate}
	\item $\overline{A} = J_2 A J_1 = (A^\hadj)^\transpose = (A^\transpose)^\hadj$,
	\item $(AB)^\transpose = B^\transpose A^\transpose$, $\overline{AB} = \overline{A} \,\overline{B}$,
	\item $(A^\transpose)^\transpose = A$, $\overline{\overline{A}} = A$,
	\item $(J_2 A J_1)^\hadj = J_1 A^\hadj J_2$, and $(J_2 A J_1)^{-1} = J_1 A^{-1} J_2$ for $A$ invertible.
\end{enumerate}
\end{lemma}

\subsection{Decomposable maps}
\label{sec:DecMaps}

Let $\mathscr{A}$, $\mathscr{B}$ be partially ordered, unital C*-algebras and denote $\mathscr{A}^+$, $\mathscr{B}^+$ the convex cones of \emph{positive elements} in $\mathscr{A}$ and $\mathscr{B}$, respectively. A bounded linear map $\phi : \mathscr{A}\to\mathscr{B}$ is called \cite{Stoermer2013}:

\begin{itemize}
	\item \emph{positive} if $\phi(\mathscr{A}^+)\subseteq \mathscr{B}^+$, i.e.~when it is order preserving,
	\item \emph{$k$-positive} if its extension $\phi_k : \operatorname{\mathbb{M}}_{k}(\mathscr{A})\to\mathbb{M}_{k}(\mathscr{B})$ defined by
\begin{equation}
	\phi_k\left([a_{ij}]_{i,j=1}^{k}\right) = \left[\phi(a_{ij})\right]_{i,j=1}^{k}, \quad a_{ij}\in\mathscr{A}
\end{equation}
is positive for given $k\in\naturals$ ($1$-positive map is simply positive), and
\item \emph{completely positive} (CP) if it is $k$-positive for all $k\in\naturals$.
\end{itemize}

Let us now take $\mathscr{B} = B(H)$ and let $\transpose$ be the transposition, $\transpose (A) = A^\transpose = J_H A^\hadj J_H$ as in Section \ref{sec:Transposition}. Then, a linear map $\phi : \mathscr{A}\to B(H)$ is additionally called:
\begin{itemize}
	\item \emph{$k$-copositive} if $\transpose\circ\phi$ is $k$-positive, and
	\item \emph{completely copositive} (coCP) if $\transpose\circ\phi$ is CP.
\end{itemize}

   We will denote sets of all CP and coCP maps from $\mathscr{A}$ to $B(H)$ by $\cp{\mathscr{A}}{H}$ and $\cocp{\mathscr{A}}{H}$, respectively. Both these are pointed convex cones in $B(\mathscr{A},B(H))$, closed with respect to the BW-topology (for a comprehensive coverage of various subclasses of positive maps on C*-algebras see e.g.~\cite{Stoermer2013}). Finally, a positive linear map $\varphi : \mathscr{A}\to B(H)$ is called \emph{decomposable} if there exist maps $\phi_1 \in \cp{\mathscr{A}}{H}$ and $\phi_2\in\cocp{\mathscr{A}}{H}$ such that
\begin{equation}\label{eq:DecomposableMap}
	\varphi = \phi_1 + \phi_2 .
\end{equation}
The set $\dec{\mathscr{A}}{H}$ of all decomposable maps is then the \emph{Minkowski sum} of cones of CP and coCP maps,
\begin{align}
	\dec{\mathscr{A}}{H} &= \cp{\mathscr{A}}{H} + \cocp{\mathscr{A}}{H} \\
	&= \conv{\cp{\mathscr{A}}{H}\cup \cocp{\mathscr{A}}{H}},\nonumber
\end{align}
where $\conv{A}$ is the \emph{convex hull of set $A$}, i.e.~the smallest convex set containing $A$. Hence, $\dec{\mathscr{A}}{H}$ is also a BW-closed pointed convex cone, and moreover it is an example of a \emph{mapping cone}, i.e.~is algebraically closed with respect to compositions with CP maps, from left and right. For maps on $\matr{n}$ we will employ a simplified, self-explanatory notation $\cpe{\matr{n}}$, $\cocpe{\matr{n}}$ and $\dece{\matr{n}}$.

Decomposable maps are characterized in terms of a dilation theorem by St{\o}rmer \cite{Stoermer2013,Stormer_1982}: for every $\varphi\in\dec{\mathscr{A}}{H}$ there exists a Hilbert space $K$, a bounded linear map $V : H\to K$ and a Jordan morphism $\pi : \mathscr{A}\to B(K)$ such that
\begin{equation}\label{eq:StoermerDilation}
	\varphi (a) = V^\hadj \pi(a) V
\end{equation}
for $a\in\mathscr{A}$. St{\o}rmer's result is a generalization of a celebrated Stinespring's dilation theorem \cite{Stinespring_1955} which characterizes, in a similar manner, all CP maps: any $\phi$ which is CP enjoys the factorized form \eqref{eq:StoermerDilation}, however with a *-homomorphism replacing Jordan morphism. Similarly, any coCP map is also in a form \eqref{eq:StoermerDilation}, but with $\pi$ being a *-antihomomorphism. Despite an enormous effort of many mathematicians the exact structure of cones of not only positive, but also decomposable maps is still far from being well-understood, even in finite-dimensional case, with one notable exception: when $\mathscr{A} = \matr{n}$, $\mathscr{B}=\matr{m}$ and $mn\leqslant 6$, then every positive map $\phi : \mathscr{A}\to\mathscr{B}$ is automatically decomposable, as was shown by Woronowicz \cite{Woronowicz1976}. In higher-dimensional algebras this is however not the case since already for $n=m=3$ there exist positive, indecomposable maps (i.e.~not in a form \eqref{eq:DecomposableMap}), as first demonstrated by Choi \cite{Choi1980} and then explored in literature.

\section{Covariant decomposable maps}
\label{sec:CovariantDecMaps}

In this section, we will focus on some aspects of covariant linear maps, which are further assumed to be decomposable in the sense of Section \ref{sec:DecMaps}. Main results of this part are presented in a form of Theorems \ref{thm:MainResult1} and \ref{thm:MainResult2} below. We divide our results into three major subsections, where the first one concerns maps on general C*-algebras and the second one is devoted to finite-dimensional case and focused on operator sum representation with respect to Frobenius Hermitian basis.

\subsection{Maps on general C*-algebras}

Let $\mathcal{G}$ be a compact Lie group and let $\phi : \mathscr{A}\to\mathscr{B}$ be a linear map between unital C*-algebras. For any $g\in\mathcal{G}$, we denote with $g(\cdot)$ the \emph{action} of $\mathcal{G}$ on $\mathscr{A}$, i.e.~we treat $g$ as a *-automorphism on $\mathscr{A}$. For $a$, $b$ in any C*-algebra we denote $\mathrm{Ad}_{a}(b) = aba^\hadj$.

\begin{definition}
Let $U : \mathcal{G}\to\mathscr{B}$ be a unitary representation of $\mathcal{G}$. If it happens that
\begin{equation}
	\phi\circ g = \mathrm{Ad}_{U(g)}\circ\phi ,
\end{equation}
or, explicitly,
\begin{equation}
	\phi (g(a)) = U(g) \phi(a) U(g)^{-1}
\end{equation}
for all $g\in \mathcal{G}$, $a\in\mathscr{A}$, then $\phi$ will be called \emph{$U$-covariant}.
\end{definition}

Let us now take $\mathscr{B} = B(H)$. For a chosen representation $U : \mathcal{G}\to B(H)$ denote by $\mathcal{C}_U$ a norm-closed linear subspace in $B(\mathscr{A},B(H))$ of all $U$-covariant maps (see Theorem \ref{thm:CuNormClosed} in the Appendix \ref{app:SideTheorems}).

\begin{theorem}\label{thm:MainResultDec}
A linear map $\varphi \in \dec{\mathscr{A}}{H}$ is $U$-covariant if and only if there exist $U$-covariant maps $\phi_1 \in \cp{\mathscr{A}}{H}$ and $\phi_2 \in \cocp{\mathscr{A}}{H}$ such that $\varphi = \phi_1 + \phi_2$.
\end{theorem}

\begin{proof}
Only the ``if'' part needs to be addressed. Since $\mathcal{G}$ is compact, it is amenable i.e.~it admits a unique translation invariant mean with respect to Haar measure $\mu$. Define a linear operator $\proj{U}$ on $B(\mathscr{A}, B(H))$ as a Bochner integral
\begin{equation}
	\proj{U}(\phi) = \int_{\mathcal{G}} \mathrm{Ad}_{U(g)}\circ\phi \circ g^{-1} \, d\mu(g).
\end{equation}
It is not hard to see that in fact $\proj{U}$ is a projection onto the subspace of $U$-covariant maps: indeed, for any $\phi$ and any $g\in\mathcal{G}$ we have, by translation invariance,
\begin{align}
	\proj{U}(\phi) \circ g &= \int_{\mathcal{G}} \mathrm{Ad}_{U(h)}\circ\phi \circ h^{-1}g\, d\mu(h) = \int_{\mathcal{G}} \mathrm{Ad}_{U(gh)}\circ\phi\circ (gh)^{-1}g\, d\mu(h) \\
	&= \int_{\mathcal{G}} \mathrm{Ad}_{U(g)}\circ \mathrm{Ad}_{U(h)}\circ\phi\circ h^{-1}g^{-1}g\, d\mu(h) = \mathrm{Ad}_{U(g)} \circ \proj{U}(\phi), \nonumber
\end{align}
since $\mu(\mathcal{G})=1$ so $\proj{U}(\phi)$ is $U$-covariant and $\im{\proj{U}} = \mathcal{C}_{U}$; also, by similar arguments we have
\begin{align}
	\proj{U}^{2}(\phi) &= \int_{\mathcal{G}} \int_{\mathcal{G}} \mathrm{Ad}_{U(gh)}\circ \phi \circ (gh)^{-1} \, d\mu(h) d\mu(g) \\
	&= \int_{\mathcal{G}} d\mu(g) \int_{\mathcal{G}} \mathrm{Ad}_{U(gg^{-1}h)}\circ \phi \circ (gg^{-1}h)^{-1} \, d\mu(h) \nonumber \\
	&= \int_{\mathcal{G}} d\mu(g) \int_{\mathcal{G}} \mathrm{Ad}_{U(h)} \circ \phi \circ h^{-1} d\mu(h) = \proj{U}(\phi), \nonumber
\end{align}
so it is a projection with range $\mathcal{C}_U$. Let $\varphi$ be decomposable and $U$-covariant, so that $\varphi = \phi_1 + \phi_2$ with $\phi_1$ being CP and $\phi_2$ coCP. Then,
\begin{equation}
	\varphi = \proj{U}(\varphi) = \proj{U}(\phi_1) + \proj{U}(\phi_2).
\end{equation}
However, by virtue of Theorem \ref{prop:ConesInvarianceP} (see Appendix \ref{app:SideTheorems}) we have $\proj{U}(\phi_1) \in \cp{\mathscr{A}}{H}$ and $\proj{U}(\phi_1) \in \cocp{\mathscr{A}}{H}$ so $\varphi$ indeed admits the claimed decomposition into a sum of $U$-covariant maps.
\end{proof}

It follows immediately that it suffices to find appropriate characterization of $U$-covariant coCP maps in order to determine all $U$-covariant decomposable maps. For a given unitary representation $U : \mathscr{A}\to B(H)$ we define its \emph{conjugate representation} $\overline{U}$ by setting
\begin{equation}
	\overline{U}(g) = \overline{U(g)},
\end{equation}
also unitary, with conjugation $A\mapsto \overline{A}$ in $B(H)$ given by $\overline{A} = (A^\hadj)^\transpose = J_H A J_H$ as in Section \ref{sec:Transposition}. A direct observation about covariance of coCP maps can be easily checked:

\begin{theorem}\label{thm:UcovUbarCov}
The following statements hold:
\begin{enumerate}
	\item A map $\phi$ is $U$-covariant if and only if $\transpose\circ\phi$ is $\overline{U}$-covariant;
	\item Conversely, $\phi$ is $\overline{U}$-covariant if and only if $\transpose\circ\phi$ is $U$-covariant;
	\item A map $\phi$ is coCP and $U$-covariant if and only if there exists a $\overline{U}$-covariant CP map $\psi$ such that $\phi = \transpose\circ\psi$.
\end{enumerate}
\end{theorem}

As the above theorem suggests, it is enough to find a characterization of $\overline{U}$-covariant CP maps. The following result, formulated in terms of a dilation theorem, allows to achieve this: 

\begin{theorem}\label{thm:MainResult1}
Let $\mathscr{A}$ be a unital C*-algebra and $H$ a Hilbert space. A linear map $\phi\in\cp{\mathscr{A}}{H}$ is $\overline{U}$-covariant if and only if there exists a Hilbert space $K$, a *-homomorphism $\pi : \mathscr{A}\to B(K)$ and a bounded linear operator $V : H\to K$ such that, for all $a\in\mathscr{A}$, $g \in \mathcal{G}$,
\begin{equation}
	\phi (a) = V^\hadj \pi (a) V
\end{equation}
and
\begin{equation}\label{eq:MainResult1eq1}
	(\pi \circ g)(a) = W(g) \pi (a)W(g)^{-1}
\end{equation}
for a representation $W : \mathcal{G}\to B(K)$ satisfying
\begin{equation}\label{eq:MainResult1eq2}
	V \overline{U}(g) = W(g)V.
\end{equation}
\end{theorem}

\begin{proof}
The theorem is precisely the Scutaru's result \cite[Theorem 1]{Scutaru1979}, however formulated for a representation $\overline{U}$ instead of $U$. For Reader's convenience we show the construction explicitly.

Let $\mathscr{A}\odot H$ be the algebraic tensor product. For any $\xi, \eta\in\mathscr{A}\odot H$ in form $\xi = \sum_{i=1}^{n} a_i \otimes w_i$, $\eta = \sum_{i=1}^{n} b_i \otimes v_i$ where $a_i, b_i \in \mathscr{A}$, $w_i, v_i \in H$, we define a positive semidefinite sesquilinear form $\iprod{\cdot}{\cdot}_\odot$ by setting
\begin{equation}\label{eq:FormAlgTensProd}
	\iprod{\xi}{\eta}_\odot = \sum_{i,j=1}^{n} \iprod{w_i}{\phi(a_{i}^{\hadj}b_j)v_j}
\end{equation}
where $\iprod{\cdot}{\cdot}$ is an inner product in $H$. Since this form is degenerate in general, it induces a \emph{seminorm} $\xi\mapsto\sqrt{\iprod{\xi}{\xi}_\odot}$ on $\mathscr{A}\odot H$. Therefore, the quotient space $(\mathscr{A}\odot H) / \mathcal{N}$, where $\mathcal{N} = \{\xi  : \iprod{\xi}{\xi}_\odot = 0\}$, is a pre-Hilbert space endowed with the inner product $\iprod{[\xi]}{[\eta]} = \iprod{\xi}{\eta}_\odot$; here $[\xi] = \xi + \mathcal{N}$ denotes the equivalence class of $\xi$ in the quotient. Then one can define $K$ as $K = \overline{(\mathscr{A}\odot H)/\mathcal{N}}^{\|\cdot\|}$, the Hilbert space completion of the quotient with respect to the norm induced by this inner product. Operator $V : H\to K$ is then
\begin{equation}
	Vx = [1_\mathscr{A} \otimes x], \quad x\in H
\end{equation}
where $1_\mathscr{A}$ is a unit in $\mathscr{A}$, while $\pi : \mathscr{A}\to B(H)$ given by
\begin{equation}
	\pi(a)[\xi] = \left[\sum_{i=1}^{n} a a_i \otimes w_i \right], \quad a\in\mathscr{A}
\end{equation}
serves as a desired *-homomorphism. Construction of $K$, $V$ and $\pi$ is then standard as in the proof of the Stinespring's theorem \cite{Stinespring_1955}. Next we define a representation $W_0 : \mathcal{G}\to B(\mathscr{A}\odot H)$ as well as its lifting $W : \mathcal{G}\to B(K)$ as
\begin{equation}\label{eq:WgDef}
	W_0 (g)\xi = \sum_{i=1}^{n} g(a_i ) \otimes \overline{U}(g) w_i , \quad W(g)[\xi] = [W_0 (g)\xi].
\end{equation}
With direct computation one confirms that
\begin{equation}
	\iprod{W_0 (g)\xi}{W_0 (g)\eta}_\odot = \sum_{i,j=1}^{n}\iprod{\overline{U}(g)w_i}{\phi(g(a_{i}^{\hadj}b_j))\overline{U}(g)v_j}_\odot = \iprod{\xi}{\eta}_\odot
\end{equation}
which comes from unitarity of $\overline{U}$, group homomorphism properties of $g$ and $\overline{U}$-covariance of $\phi$. In result, subspace $\mathcal{N}$ is $W_0 (g)$-invariant, $W_0 (g)(\mathcal{N}) = \mathcal{N}$. This allows to check
\begin{equation}
	W(g)\pi(a) W(g^{-1})[\xi] = \left[\sum_{i=1}^{n} g(a) a_i \otimes w_i\right] = (\pi\circ g)(a)[\xi]
\end{equation}
for all $g\in\mathcal{G}$, $\xi \in \mathscr{A}\odot H$. Since this equality carries over to the completion, we obtain 
\begin{equation}
	(\pi\circ g)(a) = W(g)\pi(a)W(g)^{-1}, \quad g\in\mathcal{G}, a\in\mathscr{A}
\end{equation}
i.e.~equation \eqref{eq:MainResult1eq1}. Executing similar calculation for $\pi(a)V \overline{U}(g) h$, $h\in H$, leads to
\begin{equation}\label{eq:MainResult1eq3}
	\pi(a)V \overline{U}(g) h = [a \otimes \overline{U}(g) h].
\end{equation}
Notice that ``covariance'' property \eqref{eq:MainResult1eq1} holds for \emph{any} element $g\in\mathcal{G}$, so in particular, also for $g^{-1}$; this implies
\begin{equation}
	(\pi \circ g^{-1})(a) = W(g^{-1})\pi(a) W(g),
\end{equation}
\begin{equation}
	\pi (g^{-1}(a))Vh = W(g^{-1})\pi(a) W(g)Vh,
\end{equation}
giving
\begin{align}
	\pi(a)W(g)Vh &= W(g) \pi (g^{-1}(a)) Vh = W(g)[g^{-1}(a)\otimes h]\\
	&= \left[g(g^{-1}(a))\otimes \overline{U}(g)h\right] = [a \otimes \overline{U}(g)h].\nonumber
\end{align}
This however is the same as \eqref{eq:MainResult1eq3}, i.e.~we obtain
\begin{equation}
	\pi(a) V\overline{U}(g)h = \pi(a)W(g)Vh,
\end{equation}
satisfied for all $g\in\mathcal{G}$, $h\in H$ so $V\overline{U}(g) = W(g)V$ and \eqref{eq:MainResult1eq2} is also proved.
\end{proof}

\begin{theorem}
\label{thm:coCPUcov}
A linear map $\phi \in \cocp{\mathscr{A}}{H}$ is $U$-covariant if and only if there exists a Hilbert space $K$, a *-antihomomorphism $\pi_- : \mathscr{A}\to B(K)$ and a bounded linear operator $V : H\to K$ such that, for all $a\in\mathscr{A}$, $g \in \mathcal{G}$,
\begin{equation}
	\phi (a) = V^\hadj \pi_- (a) V
\end{equation}
and $(\pi_-\circ g)(a) = \tilde{U}(g) \pi_- (a) \tilde{U}(g)^{-1}$ for a representation $\tilde{U} : \mathcal{G}\to B(K)$ satisfying $VU(g) = \tilde{U}(g)V$.
\end{theorem}

\begin{proof}
From Theorem \ref{thm:UcovUbarCov} it is clear that it suffices to combine Theorem \ref{thm:MainResult1} with appropriate transposition map. Let $\psi : \mathscr{A}\to B(H)$ be CP and $\overline{U}$-covariant, so $\psi(a) = X^\hadj \pi_+ (a) X$ where $X : H\to K$ is linear and bounded, and $\pi_+ : \mathscr{A}\to B(K)$ is a *-homomorphism for Hilbert space $K$ properly constructed, as Theorem \ref{thm:MainResult1} dictates. $K$ comes with its own antilinear isometric isomorphism $J_K : K \to K$ of complex conjugation of vectors, $J_K x = \sum_{n\in\mathcal{I}} \overline{\iprod{\kappa_n}{x}}\kappa_n$ for $x\in K$ and $\{\kappa_n\}_{n\in\mathcal{I}}$ an orthonormal basis spanning $K$; then, $A \mapsto \tau(A) = J_K A^\hadj J_K$ is a transposition on $B(K)$ (we use the same symbol $\transpose$ regardless of the underlying space). Lemma \ref{lemma:BarTransProperties} and idempotence of $J_K$ yield
\begin{align}
	(\tau\circ\psi)(a) &= J_H (X^\hadj \pi_+(a) X)^\hadj J_H = J_H X^\hadj J_K \, J_K \pi_+(a)^\hadj J_K \, J_K X J_H \\
	&= V^\hadj (\tau\circ\pi_+)(a) V = \phi(a) \nonumber
\end{align}
for $V = J_K X J_H: H\to K$, and $\tau\circ\pi_+ = \pi_- : \mathscr{A}\to B(K)$ a *-antihomomorphism, as claimed. Employing equations \eqref{eq:FormAlgTensProd} and \eqref{eq:WgDef}, homomorphism properties of $g$ and $\overline{U}$-covariance of $\psi$, we obtain
\begin{align}
	\iprod{\xi}{W_0 (g)\eta}_\odot &= \sum_{i,j=1}^{n} \iprod{w_i}{\psi(g(g^{-1}(a_i)^\hadj)b_j) \overline{U}(g)v_j} \\
	&= \sum_{i,j=1}^{n} \iprod{g^{-1}(a_i)\otimes \overline{U}(g)^\hadj w_i}{b_j \otimes v_j},\nonumber
\end{align}
after some algebra, which yields the Hermitian conjugation of $W(g)$ to be
\begin{equation}
	W(g)^\hadj [\xi] = W(g)^{-1}[\xi] = \left[\sum_{i=1}^{n} g^{-1}(a_i) \otimes \overline{U}(g)^\hadj w_i \right],
\end{equation}
i.e.~$W(g)$ is clearly unitary. Then, transposing both sides of equality \eqref{eq:MainResult1eq1} gives
\begin{align}
	(\pi_- \circ g)(a) &= J_K (W(g)^{-1})^\hadj J_K \, J_K \pi(a)^\hadj J_K \, J_K W(g)^\hadj J_K \\
	&= \overline{W}(g) \pi_- (a) \overline{W}(g)^{-1}, \nonumber
\end{align}
from which we decipher $\tilde{U} = \overline{W}$. Finally, by transposing left hand side of \eqref{eq:MainResult1eq2} with $V$ replaced by $X$ we have
\begin{align}\label{eq:trXU}
	\transpose(X\overline{U}(g)) &= J_H\overline{U}(g)^\hadj J_H \, J_H X^\hadj J_K = \left((J_K X J_H)\overline{\overline{U}(g)}\right)^\hadj \\
	&= (V U(g))^\hadj , \nonumber
\end{align}
while transpose of the right hand side is
\begin{equation}\label{eq:trWX}
	\transpose(W(g)X) = J_H X^\hadj J_K \, J_K W(g)^\hadj J_K = V^\hadj \overline{W}(g)^\hadj.
\end{equation}
Equating \eqref{eq:trXU} with \eqref{eq:trWX} yields $VU(g) = \tilde{U}(g)V$, i.e.~$\tilde{U} = \overline{W}$ satisfies the required equation and the proof is finished.
\end{proof}

As a kind of a corollary, we present the following dilation theorem characterizing $U$-covariant decomposable maps:

\begin{theorem}\label{thm:MainResult2}
Let $\mathscr{A}$ be a unital C*-algebra and $H$ a Hilbert space. A linear map $\varphi \in \dec{\mathscr{A}}{H}$ is $U$-covariant if and only if there exists a Hilbert space $K$, a Jordan morphism $\pi : \mathscr{A}\to B(K)$ and a bounded linear operator $V : H\to K$ such that, for all $a\in\mathscr{A}$ and $g \in \mathcal{G}$,
\begin{equation}\label{eq:MainResult2Eq1}
	\varphi (a) = V^\hadj \pi (a) V
\end{equation}
and
\begin{equation}\label{eq:MainResult2Eq2}
	(\pi \circ g)(a) = Z(g) \pi (a)Z(g)^{-1}
\end{equation}
for a representation $Z : \mathcal{G}\to B(K)$ satisfying
\begin{equation}\label{eq:MainResult2Eq3}
	V U(g) = Z(g)V.
\end{equation}
\end{theorem}

\begin{proof}
Let $\varphi : \mathscr{A}\to B(H)$ be decomposable and $U$-covariant. Theorem \ref{thm:MainResultDec} guarantees that there exist a CP map $\phi_1$ and a coCP map $\phi_2$, both $U$-covariant, such that $\varphi = \phi_1 + \phi_2$. Scutaru's result \cite[Theorem 9]{Scutaru1979} and Theorem \ref{thm:MainResult1} infer existence of Hilbert spaces $K_1$, $K_2$, bounded linear operators $V_1 : H \to K_1$, $V_2 : H \to K_2$, a *-homomorphism $\pi_1 : \mathscr{A}\to B(K_1)$ and a *-antihomomorphism $\pi_2 : \mathscr{A}\to B(K_2)$, such that 
\begin{subequations}
	\begin{equation}
		\phi_i (a) = V_{i}^{\hadj} \pi_i (a) V_i,
	\end{equation}
	\begin{equation}
		(\pi_i\circ g)(a) = \tilde{U}(g) \pi_i (a) \tilde{U}(g)^{-1},
	\end{equation}
	\begin{equation}
		V_i U(g) = \tilde{U}(g) V_i ,
	\end{equation}
\end{subequations}
for $i=1,2$ and for all $a\in\mathscr{A}$, $g\in\mathcal{G}$. Following St{\o}rmer \cite{Stoermer2013,Stormer_1982}, let us then take $K = K_1 \oplus K_2$ and define $V$ and $\pi$ by
\begin{subequations}
	\begin{equation}
		V: H \to K_1 \oplus K_2, \quad V = V_1 + V_2, 
	\end{equation}
	\begin{equation}
		\pi : \mathscr{A}\to B(K_1 \oplus K_2), \quad \pi(a) = \proj{1} \pi_1 (a) \proj{1} + \proj{2} \pi_2 (a) \proj{2},
	\end{equation}
\end{subequations}
where $\proj{1}$, $\proj{2}$ are orthogonal projections onto $K_1$, $K_2$.  By direct check, $a\mapsto\proj{1} \pi_1 (a) \proj{1}$ and $a\mapsto\proj{2} \pi_2 (a) \proj{2}$ are a *-homomorphism and *-antihomomorphism, respectively, and so $\pi$ is a Jordan morphism. Therefore, since $\im{\proj{i}} = \im{V_i} = K_i$, we can put $V_i = \proj{i}V_i$ and
\begin{align}
	\varphi(a) &= V^\hadj \pi(a) V \\
	&= (V_{1}^{\hadj}\proj{1} + V_{2}^{\hadj}\proj{2})(\proj{1}\pi_1 (a) \proj{1} + \proj{2}\pi_2 (a) \proj{2})(\proj{1}V_1 + \proj{2}V_2) \nonumber \\
	&= V_{1}^{\hadj} \pi_1 (a) V_1 + V_{2}^{\hadj}\pi_2 (a) V_2 \nonumber\\
	&= \phi_1 (a) + \phi_2 (a) , \nonumber
\end{align}
so $\varphi$ has the claimed form \eqref{eq:MainResult2Eq1}. Let $Z : \mathcal{G}\to B(K_1\oplus K_2)$ be given by
\begin{equation}
	Z(g) = \proj{1} \tilde{U}(g) \proj{1} + \proj{2} \tilde{U}(g) \proj{2}.
\end{equation}
One easily confirms $Z(g)^{-1} = \proj{1} \tilde{U}(g)^{-1} \proj{1} + \proj{2} \tilde{U}(g)^{-1} \proj{2}$. Then, a simple algebra involving orthogonality of projections $\proj{i}$ shows
\begin{equation}
	Z(g) \pi(a) Z(g)^{-1} = (\pi \circ g)(a),
\end{equation}
i.e.~\eqref{eq:MainResult2Eq2} is satisfied. Finally,
\begin{align}
	Z(g)V &= (\proj{1} \tilde{U}(g) \proj{1} + \proj{2} \tilde{U}(g) \proj{2})(\proj{1}V_1 + \proj{2}V_2) \\
	&= \proj{1}\tilde{U}(g)V_1 + \proj{2}\tilde{U}(g)V_2 \nonumber \\
	&= \proj{1}V_1 U(g) + \proj{2}V_2 U(g)  \nonumber \\
	&= VU(g),\nonumber
\end{align}
since $\tilde{U}(g)V_i = V_i U(g)$, so \eqref{eq:MainResult2Eq3} also holds.
\end{proof}

\subsection{Maps on \texorpdfstring{$B(H)$}{B(H)}}

As a particular example, consider a von Neumann algebra $\mathscr{A} = B(H)$. Let $\varphi : B(H)\to B(H)$ be decomposable and $U$-covariant; then, Theorems \ref{thm:MainResultDec} and \ref{thm:UcovUbarCov} yield existence of maps $\phi , \psi \in \cpe{B(H)}$ such that $\varphi = \phi + \transpose\circ\psi$ with $\phi$ being $U$-covariant and $\psi$ being $\overline{U}$-covariant. Employing Kraus lemma \cite{Kraus1971}, map $\varphi$ may be further written as
\begin{equation}
	\varphi(A) = \sum_n X_n A X_{n}^{\hadj} + \sum_{n} W_n A^\transpose W_{n}^{\hadj}, \quad A\in B(H),
\end{equation}
for some countable families of bounded operators $(X_n)$, $(W_n)$ on $H$ such that $\sum_n X_n X_{n}^{\hadj}$ and $\sum_n W_n W_{n}^{\hadj}$ converge in weak-* topology in $B(H)$. For $\mathcal{G}$ a compact Lie group, let $g$ be definad via
\begin{equation}
	g(A) = V(g)AV(g)^{-1}
\end{equation}
where $V(g)\in B(H)$ is unitary, so $g$ is also a unitary representation of $\mathcal{G}$. The covariance condition then reads
\begin{equation}
	\varphi\circ\mathrm{Ad}_{V(g)} =  \mathrm{Ad}_{U(g)}\circ \varphi , \quad g\in\mathcal{G}
\end{equation}
which translates to conditions for the CP and coCP parts
\begin{align}
	\phi(A) &= U(g) \phi \left(V(g)^{-1}AV(g)\right)U(g)^{-1}, \label{eq:BHcondCP}\\ 
	\psi(A) &= \overline{U}(g) \psi \left(V(g)^{-1}AV(g)\right)\overline{U}(g)^{-1} \label{eq:BHcondcoCP}
\end{align}
being satisfied for all $A\in B(H)$, $g\in\mathcal{G}$.

\begin{theorem}
	A decomposable map $\varphi$ on $B(H)$ is $U$-covariant if and only if there exist families $(X_n)$, $(W_n)$ in $B(H)$ such that $\sum_n X_n X_{n}^{\hadj}$, $\sum_n W_n W_{n}^{\hadj}$ are weakly-* convergent in $B(H)$ and
	\begin{equation}
		U(g) X_n V(g)^{-1} = \sum_m N_{nm} X_m, \quad \overline{U}(g) W_n V(g)^{-1} = \sum_m M_{nm} W_m ,
	\end{equation}
	for some unitary matrices $[N_{nm}]$, $[M_{nm}]$.
\end{theorem}

\begin{proof}
	It is enough to notice that covariance conditions \eqref{eq:BHcondCP} and \eqref{eq:BHcondcoCP} yield, with assumption that Stinespring representations involved are minimal,
	\begin{align}
		\phi(A) &= \sum_n U(g) X_n V(g)^{-1} A V(g) X_{n}^{\hadj} U(g)^{-1}, \\
		\psi(A) &= \sum_n \overline{U}(g) W_n V(g)^{-1} A V(g) W_{n}^{\hadj} \overline{U}(g)^{-1}
	\end{align}
	where both series converge in weak-* topology in $B(H)$ for every $A$, i.e.~they represent CP maps via new Kraus operators $\tilde{X}_n = U(g) X_n V(g)^{-1}$, $\tilde{W}_n = \overline{U}(g) W_n V(g)^{-1}$. Since minimal Stinespring representations of the same map are unique up to a unitary transformation, the claim follows.
\end{proof}

\subsection{Maps covariant w.r.t.~the maximal commutative subgroup of \texorpdfstring{$\mathrm{U}(n)$}{U(n)}}
\label{sec:MapsOnMn}

In this section we focus our attention on maps on $\matr{n}$ subject to a very particular covariance condition. Namely, let $\mathscr{A}=\matr{n}$ and let $\mathcal{G} = \mathbb{S}^1 \times ... \times \mathbb{S}^1$ for $\mathbb{S}^1$ a circle group, so that any element $g \in \mathcal{G}$ is of a form $g = (g_j)_{j=1}^{n} = (e^{ix_1},...,e^{ix_n})$ for some real vector $(x_j)\in\reals^n$. Group structure of $\mathcal{G}$ is defined element-wise, i.e. for $g, h\in\mathcal{G}$ we have $gh = (g_j h_j)$, $g^{-1} = (g_{j}^{-1}) = (e^{-ix_j})$ and $(1, ... ,1)$ is a neutral element. We define $U : \mathcal{G}\to \mathrm{U}(n)$ as
\begin{equation}\label{eq:UmaxComm}
	U(g) = \sum_{j=1}^{n} e^{i x_j} E_{jj} , \quad g = (e^{ix_j}),
\end{equation}
i.e.~$U(\mathcal{G})$ is the \emph{maximal commutative subgroup} of all unitary diagonal matrices of size $n$. For the automorphism $g : \matr{n}\to\matr{n}$, we take it in a form $g(A) = U(g)AU(g)^{-1}$. In other words, we will be interested in all $U$-covariant maps $\varphi : \matr{n}\to\matr{n}$ subject to equality
\begin{equation}\label{eq:CovarianceCondition}
	\varphi \left(U(g) A U(g)^{-1}\right) = U(g) \varphi(A) U(g)^{-1}
\end{equation}
for all $g = (e^{ix_j})$, $(x_j)\in\reals^n$ and all $A\in\matr{n}$; we will simply refer to them as \emph{covariant} throughout this section. 

\subsubsection{Conditions for covariance}

One shows $\phi$ is covariant if and only if \cite[Proposition 28]{Chruscinski2022}
\begin{equation}\label{eq:CovariantChr}
	\phi(A) = \sum_{i,j=1}^{n} a_{ij}(1-\delta_{ij}) E_{ij} A E_{ji} + \sum_{i,j=1}^{n} b_{ij} E_{ii} A E_{jj}
\end{equation}
for some $a_{ij},b_{ij}\in\complexes$. Further, $\phi$ is \emph{Hermiticity preserving}, $\phi(A)^\hadj = \phi(A^\hadj)$, if and only if $a_{ij}\in\reals$ and $[b_{ij}]$ is Hermitian. Also, $\phi\in\cpe{\matr{n}}$ if and only if $a_{ij}\geqslant 0$ and $[b_{ij}]\in\matr{n}^+$. Equation \eqref{eq:CovariantChr} is a complete characterization of the so-called \emph{operator sum representation} of covariant maps, expressed in terms of a canonical basis in $\matr{n}$. In this section however we find such characterization also with respect to the \emph{Frobenius basis} $\{F_i\}$. That said, we focus on the operator sum representation
\begin{equation}\label{eq:phiA}
	\phi(A) = \sum_{i,j=1}^{n^2} c^{\phi}_{ij} F_i A F_{j}^{\hadj}
\end{equation}
and internal structure of the matrix $\hat{c}_\phi = [c^{\phi}_{ij}] \in \matr{n^2}$ which guarantees the covariance of $\phi$; note that $[c^{\phi}_{ij}]$ uniquely determines the map $\phi$, up to the choice of a basis in $\matr{n}$ \cite{Szczygielski2023}. For this we first formulate the following sufficient and necessary condition for covariance in a slightly more general manner than we need.

Let $U,V : \mathcal{G}\to\matr{n}$ be (any) two unitary representations. We say a linear map $\phi : \matr{n} \to \matr{n}$ is \emph{$(U,V)$-covariant} if
\begin{equation}\label{eq:CovarianceConditionUV}
	\phi\circ\mathrm{Ad}_{U(g)} = \mathrm{Ad}_{V(g)}\circ\phi
\end{equation}
for all $g\in\mathcal{G}$. Also, let us define a matrix-valued function $\hat{\alpha}_{U,V} : \mathcal{G}\to\matr{n^2}$ by
\begin{equation}
	\hat{\alpha}_{U,V}(g) = [\alpha^{U,V}_{ij}(g)], \quad \alpha^{U,V}_{ij}(g) = \hsiprod{F_i}{F_{j}^{\prime}(g)},
\end{equation} 
for
\begin{equation}
	F_{i}^{\prime}(g) = V(g)^{-1}F_i U(g)
\end{equation}
constituting for a new orthonormal basis in $\matr{n}$. We have the following result:

\begin{proposition}\label{prop:CovUVcommAlpha}
A linear map $\phi : \matr{n}\to\matr{n}$ is $(U,V)$-covariant if and only if $\hat{c}_\phi$ commutes with $\hat{\alpha}_{U,V}(g)$ for all $g\in\mathcal{G}$. In result, a linear subspace $\mathcal{C}_{U,V}$ of all $(U,V)$-covariant maps on $\matr{n}$ is isomorphic to $\hat{\alpha}_{U,V}(\mathcal{G})^\prime$.
\end{proposition}

\begin{proof}
The covariance condition \eqref{eq:CovarianceConditionUV} yields
\begin{equation}\label{eq:PhiAdPhiAd}
	\phi = \mathrm{Ad}^{-1}_{V(g)} \circ \phi \circ \mathrm{Ad}_{U(g)}
\end{equation}
for all $g \in \mathcal{G}$. Putting \eqref{eq:phiA} into \eqref{eq:PhiAdPhiAd} and then expanding new matrices $F_{i}^{\prime}(g)$ in the Frobenius basis, $F_{i}^{\prime}(g) = \sum_{j}\hsiprod{F_j}{F_{i}^{\prime}(g)}F_j$, we deduce
\begin{equation}
	c^{\phi}_{ij} = \sum_{kl} \alpha^{U,V}_{ik}(g) c^{\phi}_{kl} \overline{\alpha^{U,V}_{jl}(g)}, \quad \text{or} \quad \hat{c}_\phi = \hat{\alpha}_{U,V}(g) \hat{c}_\phi {\hat{\alpha}(g)}_{U,V}^{\hadj}
\end{equation}
after easy algebra. Since $\hat{\alpha}_{U,V}(g)$ is clearly unitary, $\hat{c}_\phi$ and $\hat{\alpha}_{U,V}(g)$ commute. The claim then follows after noting that $\hat{c}_\phi$ identifies $\phi$ uniquely.
\end{proof}

For covariance w.r.t.~the diagonal unitary matrices as discussed here, $U = V$ is given by \eqref{eq:UmaxComm}. The subspace $\mathcal{C}_U$ of all covariant maps is then determined by structure of $\hat{\alpha}_{U,U}(g) \equiv \hat{\alpha}_U (g)$ and appropriate commutant can be found with a bit of effort: it is shown in the Proposition \ref{prop:AlphaStructure} (Appendix \ref{app:SideTheorems}) that each $\hat{\alpha}_U(g)$ has the following block structure
\begin{equation}\label{eq:AlphaGDef}
	\hat{\alpha}_U(g) = \left( \begin{array}{ccc} \cos{\Delta} & -\sin{\Delta} & 0 \\ \sin{\Delta} & \cos{\Delta} & 0 \\ 0 & 0 & I_n \end{array}\right),
\end{equation}
where $I_n \in \matr{n}$ is an identity matrix, and
\begin{equation}\label{eq:DeltaDef}
	\Delta = \operatorname{diag}{\{x_{\mu(k)}-x_{\nu (k)}\}}, \quad 1\leqslant k \leqslant \binom{n}{2}
\end{equation}
is determined by $g$; here notation $\mu(k)$, $\nu(k)$ denotes a bijection $k\mapsto (\mu,\nu)$ which defines an ordering of basis $F_i$ as mentioned in Section \ref{sec:FrobeniusBasis}.

Let $z\in \complexes\setminus\{0\}$. We say that the set $S\subset\complexes$ is \emph{$z$-congruence free} when no two different $z_1$, $z_2$ exist in $S$ such that $z_1 - z_2 = k z$ for any $k\in\integers$, or when $S + (\integers \setminus\{0\}) z$ has no intersection with $S$. We make the following observation:

\begin{proposition}\label{prop:Ag}
For every $g\in\mathcal{G}$ there exists an anti-Hermitian matrix $A(g)$ such that $\hat{\alpha}_U(g) = e^{A(g)}$ and $\spec{A(g)}$ is $2\pi i$-congruence free.
\end{proposition}

\begin{proof}
Restricting the mapping $\reals^n \mapsto g = (e^{ix_j})_{j=1}^{n}$ to the hypercube $[0,2\pi )^n$ makes the mapping bijective, so it is always possible to represent matrices $U(g)$ by vectors $(x_j)$ such that $0 \leqslant x_j < 2\pi$. Matrix $\hat{\alpha}_U(g)$ given in \eqref{eq:AlphaGDef} almost resembles the rotation matrix so we can deduce the form of $A(g)$,
\begin{equation}
	A(g) = \left( \begin{array}{ccc} 0 & -\Delta & 0 \\ \Delta & 0 & 0 \\ 0 & 0 & 0 \end{array}\right).
\end{equation}
Naturally, $(x_j)\in [0,2\pi )^n$ implies $r(\Delta) < 2\pi$, with $r(T) = \max_{\lambda\in\spec{T}}{|\lambda|}$ being the spectral radius of $T$. One then confirms with power series expansion that indeed $e^{A(g)} = \hat{\alpha}_U(g)$ so $A(g)$ is the matrix logarithm of $\hat{\alpha}_U(g)$, clearly anti-Hermitian. The non-zero eigenvalues $\lambda$ of $A(g)$ satisfy
\begin{equation}
	0 = \det{\left(\begin{array}{cc}-\lambda I & -\Delta \\ \Delta & -\lambda I\end{array}\right)} = \det{(\lambda^2 I + \Delta^2)},
\end{equation}
so $\lambda\in\spec{A(g)}$ if and only if $\lambda^2 \in \spec{-\Delta^2}$. This means all $\lambda$ must be of a form
\begin{equation}
	\lambda = \pm i |x_j - x_k|
\end{equation}
for some $j<k$. However
\begin{equation}
	r(A(g)) = \max_{j,k}{|x_j - x_k|} = r(|\Delta|) < 2\pi ,
\end{equation}
so there exists no pair $(\lambda_1,\lambda_2)$ of eigenvalues in $\spec{A(g)}$ such that $\lambda_1 - \lambda_2 = 2k\pi i$ for $k\in\integers\setminus\{0\}$ i.e.~$\spec{A(g)}$ is $2\pi i$-congruence free.
\end{proof}

The $2\pi i$-congruence freedom of $\spec{A(g)}$ allows to simplify calculations significantly:

\begin{proposition}
Let $A(g)$ be as in Proposition \ref{prop:Ag}. Then $\comm{\hat{c}_\phi}{\hat{\alpha}_U(g)} = 0$ if and only if $\comm{\hat{c}_\phi}{A(g)}=0$. 
\end{proposition}

\begin{proof}
It is proved in Theorem \ref{thm:CommFunctionsOperatorsExp} (Appendix \ref{app:SideTheorems}), that for any two bounded linear operators $A,B$ on a Banach space such that $\spec{A}$ is $2\pi i$-congruence free, we have $\comm{e^A}{B} = 0$ if and only if $\comm{A}{B}=0$. The claim then follows by Proposition \ref{prop:Ag}.
\end{proof}

With the $2\pi i$-congruence freedom of the spectrum in our disposal, finding the commutant of $\hat{\alpha}_U(\mathcal{G})$ is fairly easy:

\begin{lemma}
$M\in\matr{n^2}$ satisfies $\comm{M}{\hat{\alpha}_U(g)}=0$ for all $g\in\mathcal{G}$ if and only if
\begin{equation}\label{eq:Mform}
	M = \left(\begin{array}{ccc}A_1 & -A_2 & 0 \\ A_2 & A_1 & 0 \\ 0 & 0 & A_3\end{array}\right)
\end{equation}
for arbitrary matrices $A_1, A_2 \in \operatorname{Diag}_{n(n-1)/2}(\complexes)$ and $A_3 \in \matr{n}$.
\end{lemma}

\begin{proof}
Let $M = [M_{ij}]$ where sizes of blocks $M_{ij}$ correspond with those of $\hat{\alpha}_U(g)$. Equation $\comm{M}{\hat{\alpha}_U(g)}=0$ yields a system of matrix equations
\begin{subequations}
	\begin{equation}
		M_{11}\Delta = \Delta M_{22}, \quad \Delta M_{11} = M_{22}\Delta,
	\end{equation}
	\begin{equation}
		M_{12}\Delta = -\Delta M_{21}, \quad \Delta M_{12} = -M_{21}\Delta,
	\end{equation}
	\begin{equation}
		\Delta M_{23} = \Delta M_{13} = M_{32}\Delta = M_{31}\Delta = 0,
	\end{equation}
\end{subequations}
which must hold for all possible $\Delta$, specified again by \eqref{eq:DeltaDef}. In particular, when $x_j$ are chosen all distinct, $\Delta$ is invertible and so $M_{23}$, $M_{32}$, $M_{13}$ and $M_{31}$ must all vanish. Moreover, we must have, for $\Delta$ invertible,
\begin{subequations}
\begin{equation}
	M_{11} = \Delta M_{22}\Delta^{-1} = \Delta^{-1}M_{22}\Delta,
\end{equation}
\begin{equation}
	M_{22} = \Delta M_{11}\Delta^{-1} = \Delta^{-1}M_{11}\Delta,
\end{equation}
\begin{equation}
	M_{12} = -\Delta M_{21}\Delta^{-1} = -\Delta^{-1}M_{21}\Delta,
\end{equation}
\begin{equation}
	M_{21} = -\Delta M_{12}\Delta^{-1} = -\Delta^{-1}M_{12}\Delta,
\end{equation}
\end{subequations}
which yield $\comm{\Delta^2}{M_{ij}} = 0$, or $M_{ij} = \Delta^2 M_{ij} \Delta^{-2}$, for all $i,j\in\{1,2\}$. Let us denote $M_{ij} = [m^{(i,j)}_{kl}]_{kl}$. Since $\Delta^{n}$ is diagonal for $n\in\integers$, equality $M_{ij} = \Delta^2 M_{ij} \Delta^{-2}$ yields
\begin{equation}\label{eq:MijklEq}
	m^{(i,j)}_{kl} = \left(\frac{x_{\mu(k)}-x_{\nu(k)}}{x_{\mu(l)}-x_{\nu(l)}}\right)^2 m^{(i,j)}_{kl},
\end{equation}
for all $k\neq l$. For \eqref{eq:MijklEq} to hold for all possible choices of vectors $(x_j)$ we need $m^{(i,j)}_{kl} = 0$, i.e.~all blocks $M_{ij}$ must be diagonal. This in turn yields $M_{11} = M_{22}$, $M_{12}=-M_{21}$ and $M_{33}$ can be arbitrary. This shows $M$ indeed has the claimed form.
\end{proof}

Clearly, decomposable maps require conditions for both covariance and conjugate covariance by Theorems \ref{thm:MainResultDec} and \ref{thm:UcovUbarCov}, so let us now focus on a conjugate covariant case. We will call a map $\phi : \matr{n}\to\matr{n}$ \emph{conjugate covariant} if and only if it is $(U,\overline{U})$-covariant, i.e.~when it satisfies
\begin{equation}\label{eq:PhiConjCov}
	\phi \circ \mathrm{Ad}_{U(g)}= \mathrm{Ad}_{\overline{U}(g)}\circ\phi
\end{equation}
for all $g\in\mathcal{G}$, where $\overline{U}(g) = \sum_{j=1}^{n} e^{-i x_j} E_{jj}$. The matrix $\hat{\beta}_U(g)\equiv \hat{\alpha}_{U,\overline{U}}(g)$ in this case admits a structure (see again Proposition \ref{prop:AlphaStructure} in Appendix \ref{app:SideTheorems})
\begin{equation}
	\hat{\beta}_U(g) = \left( \begin{array}{ccc} e^{i\Theta} & 0 & 0 \\ 0 & e^{i\Theta} & 0 \\ 0 & 0 & R(g) \end{array}\right),
\end{equation}
where
\begin{align}
	\Theta &= \operatorname{diag}{\{x_{\mu(j)}+x_{\nu(j)}\}}, \quad 1\leqslant j \leqslant \binom{n}{2}\\
	R (g) &= [\langle F^{\mathrm{d}}_{i},{F^{\mathrm{d}}_{j}}^{\prime}\rangle_2]_{ij}, \quad 1\leqslant i,j\leqslant n,\label{eq:ThetaRdef}
\end{align}
with $R(g)$ symmetric and unitary. And again, determining a general structure of $\hat{c}_\phi$ can be achieved with a bit of algebraic work, in a course of solving a number of linear matrix equations:

\begin{lemma}\label{lemma:MformConjCov}
$M\in\matr{n^2}$ satisfies $\comm{M}{\hat{\beta}_U(g)}=0$ for all $g\in\mathcal{G}$ if and only if
\begin{equation}\label{eq:Mform2}
	M = \left(\begin{array}{ccc}A_{11} & A_{12} & 0 \\ A_{21} & A_{22} & 0 \\ 0 & 0 & A_{33}\end{array}\right)
\end{equation}
for $A_{ij} \in \operatorname{Diag}_{n(n-1)/2}(\complexes)$ arbitrary and $A_{33} = [\kappa_{rs}]\in\matr{n}$ symmetric and specified by conditions
\begin{subequations}\label{eq:A33rs}
	\begin{align}
		\kappa_{rs} &= \frac{1}{\sqrt{rs(r+1)(s+1)}}\left( \sum_{i=1}^{r}a_i - r a_{r+1} \right) \quad \text{for } r<s<n,\\
		\kappa_{rn} &= \frac{1}{\sqrt{nr(r+1)}}\left( \sum_{i=1}^{r}a_i - r a_{r+1} \right) \quad \text{for } r<n,\\
		\kappa_{rr} &= \frac{1}{r(r+1)} \left( \sum_{i=1}^{r}a_i + r^2 a_{r+1} \right) \quad \text{for } r<n,\\
		\kappa_{nn} &= \frac{1}{n} \sum_{i=1}^{n}a_i,
	\end{align}
\end{subequations}
for arbitrary sequence $(a_i) \in \complexes^n$.
\end{lemma}

\begin{proof}
Denote again $M = [M_{ij}]$ with blocks $M_{ij}$ corresponding to those in $\hat{\beta}_U(g)$. Equation $M = \hat{\beta}_U(g) M \hat{\beta}_U(g)^\hadj$ yields a system of matrix equations
\begin{subequations}
\begin{equation}\label{eq:MijCov}
	e^{i\Theta} M_{ij} e^{-i\Theta} = M_{ij}, \quad 1\leqslant i,j \leqslant 2,
\end{equation}
\begin{equation}\label{eq:Mi3Cov}
	e^{i\Theta} M_{i3} R(g)^\hadj = M_{i3}, \quad R(g) M_{3i} e^{-i\Theta} = M_{3i}, \quad 1\leqslant i \leqslant 2,
\end{equation}
\begin{equation}\label{eq:M33Cov}
	R(g) M_{33} R(g)^\hadj = M_{33}.
\end{equation}
\end{subequations}
For case $i,j \leqslant 2$ let $M_{ij} = [m^{(i,j)}_{kl}]_{kl}$ and notice \eqref{eq:MijCov} yields
\begin{equation}
	\left(1-e^{i(x_{\mu(k)}+x_{\nu(k)}-x_{\mu(l)}-x_{\nu(l)})}\right)m^{(i,j)}_{kl} = 0
\end{equation}
which must hold for all possible choices of $(x_j)$. Since this is clearly true for $k=l$, we infer all off-diagonal terms $m^{(i,j)}_{kl}$, $k\neq l$, vanish, i.e.~all blocks $M_{ij}$ for $i,j\leqslant 2$ are diagonal. For \eqref{eq:Mi3Cov}, notice that both those matrix equations are in fact eigenequations for linear mappings 
\begin{equation}
	T_g (X) = e^{i\Theta}XR(g)^\hadj, \quad S_g (X) = R(g)Xe^{-i\Theta},
\end{equation}
for eigenvalue $1$. Recall that every such mapping $X \mapsto AXB$, where $ X,A,B\in\matr{n}$, can be isomorphically represented as a linear operator $A\otimes B^\transpose : \complexes^{n^2}\to\complexes^{n^2}$ under the transformation of \emph{vectorization} (or \emph{flattening}) of matrices \cite{Bengtsson2017}. Hence, \eqref{eq:Mi3Cov} is equivalent to the pair of mutually adjoint eigenequations
\begin{equation}
	\left(e^{i\Theta}\otimes R(g)^\hadj\right)x = x, \quad \left(R(g)\otimes e^{-i\Theta}\right)x = x.
\end{equation}
However,
\begin{align}
	\spec{T_g - \id{}} &= \spec{\left(e^{i\Theta}\otimes R(g)^\hadj - I_{n^2}\right)} \\
	&= -1+\{e^{-i(x_{\mu(k)}+x_{\nu(k)}-2x_j)}\} \nonumber
\end{align}
and one can always find such vector $(x_j)$, and such $g_0 \in\mathcal{G}$, that $\spec{T_{g_0} - \id{}}$ does not contain $0$, i.e.~the resolvent operator $(T_{g_0} - \id{})^{-1}$ exists and $T_{g_0}$ has no eigenvalue $1$. The equation $T_{g_0} (M_{i3}) = M_{i3}$ can be therefore satisfied only when $M_{i3}=0$. Since matrix $M$ is expected to satisfy \eqref{eq:Mi3Cov} for all $g\in\mathcal{G}$, we infer $M_{i3} = 0$ is the only case. The same argument then applies to $S_g$ as its spectrum is simply a complex conjugate of $\spec{T_g}$ and so $M_{3i} = 0$ as well.

For the remaining equation \eqref{eq:M33Cov}, note that \eqref{eq:ThetaRdef} implies $R(g)$ is the matrix of a linear map $\mathcal{U}_g (A)= U(g)A U(g)$ restricted to a subspace $\mathrm{Diag}_{n}(\complexes)$, with matrix elements computed in basis $\{F_{i}^{\mathrm{d}}\}$. It is straightforward to check that the spectrum of this map is
\begin{equation}
	\spec{\left.\mathcal{U}_g\right|_{\mathrm{Diag}_{n}(\complexes)}} = \spec{R(g)} = \{e^{2i x_j}, 1\leqslant j \leqslant n\},
\end{equation}
its eigenvectors are $E_{ii}$ and in result, eigenvectors $r_i$ of $R(g)$ are
\begin{equation}\label{eq:rj}
	r_i = \sum_{k=1}^{n} \hsiprod{F_{k}^{\mathrm{d}}}{E_{ii}} e_k = \sum_{k=1}^{n} (F_{k}^{\mathrm{d}})_{ii}e_k ,
\end{equation}
for $e_k$ the standard basis spanning $\complexes^n$ and $(F_{k}^{\mathrm{d}})_{ii}$ the $i$-th diagonal element of $F_{k}^{\mathrm{d}}$. The equation \eqref{eq:M33Cov} is an eigenequation for a map $\mathrm{Ad}_{R(g)}$ for eigenvalue $1$. Again, with vectorization one then confirms that
\begin{equation}
	\spec{\mathrm{Ad}_{R(g)}} = \spec{R(g)\otimes R(g)^\hadj} = \{1, e^{\pm 2i(x_j - x_k)} : 1\leqslant j < k\leqslant n\}
\end{equation}
and so $1\in\spec{\mathrm{Ad}_{R(g)}}$ for all $g\in\mathcal{G}$ and eigenequation $\mathrm{Ad}_{R(g)}(X) = R(g)XR(g)^\hadj$ always has a nontrivial solution. Since $\{r_j\}$ is an orthonormal basis spanning $\complexes^n$, we can always postulate $X$ in a general form
\begin{equation}
	X = \sum_{i,j=1}^{n} x_{ij} \iprod{r_j}{\cdot}r_i, \quad x_{ij}\in\complexes.
\end{equation}
Substituting $X$ into the eigenequation with $R(g)$, $R(g)^\hadj$ expressed by their spectral decompositions we obtain
\begin{equation}
	X = \sum_{j,k=1}^{n} e^{2i(x_j - x_k)}x_{jk} \iprod{r_k}{\cdot}r_j ,
\end{equation}
which must be satisfied for all possible choices of $(x_j)$. This however is again guaranteed only if off-diagonal elements $x_{ij} = 0$, i.e.~$X$ has a diagonal form
\begin{equation}\label{eq:XmatrixDiagonal}
	X = \sum_{i=1}^{n} a_{i} \iprod{r_i}{\cdot}r_i ,
\end{equation}
for arbitrary coefficients $a_i \in\complexes$. This can be further rewritten with \eqref{eq:rj} as
\begin{equation}
	X = \sum_{i,j,k=1}^{n} a_i (F^{\mathrm{d}}_{j})_{ii} (F^{\mathrm{d}}_{k})_{ii} E_{jk},
\end{equation}
where from \eqref{eq:FiDiag} we have
\begin{equation}
	(F^{\mathrm{d}}_{k})_{ii} = \frac{1}{\sqrt{k(k+1)}}\left( \sum_{j=1}^{k}\delta_{ij} - k\delta_{k+1,k+1} \right).
\end{equation}
Now, with a bit of algebra, one confirms all matrix elements of $X$ are indeed of a form \eqref{eq:A33rs}, so one identifies $A_{33}=X$ and the proof is concluded.
\end{proof}

As a corollary of this section we formulate the following characterization of operator sum representation of a map with respect to the Frobenius basis, for both covariance and conjugate covariance:

\begin{theorem}\label{thm:PhiUconjcovariant}
A linear map $\phi : \matr{n}\to\matr{n}$ is:
\begin{enumerate}
	\item \emph{covariant w.r.t.~the diagonal unitaries} if and only if $\hat{c}_\phi$ admits a form \eqref{eq:Mform}, i.e.
	\begin{equation}\label{eq:MatrixCgenStructure}
		\hat{c}_\phi = \left(\begin{array}{ccc}C_1 & -C_2 & 0 \\ C_2 & C_1 & 0 \\ 0 & 0 & C_3\end{array}\right),
	\end{equation}
	for arbitrary matrices $C_1 , C_2\in \operatorname{Diag}_{n(n-1)/2}{(\complexes)}$ and $C_3 \in \matr{n}$;
	\item \emph{conjugate covariant w.r.t.~the diagonal unitaries} if and only if $\hat{c}_\phi$ admits the form elaborated in Lemma \ref{lemma:MformConjCov}, i.e.
	\begin{equation}\label{eq:MatrixCgenStructure2}
		\hat{c}_\phi = \left(\begin{array}{ccc}C_{11} & C_{12} & 0 \\ C_{21} & C_{22} & 0 \\ 0 & 0 & C_{33}\end{array}\right)
	\end{equation}
	for arbitrary matrices $C_{ij}\in \operatorname{Diag}_{n(n-1)/2}{(\complexes)}$, $1\leqslant i,j\leqslant 2$, and $C_{33} = [\kappa_{ij}]\in \matr{n}$ given by equations \eqref{eq:A33rs} for arbitrary sequence $(a_i)\in\complexes^n$.
\end{enumerate}
\end{theorem}

\subsubsection{Decomposable maps}

Recall, that a linear map $\phi : \matr{n}\to\matr{n}$ is \emph{Hermiticity preserving} if and only if $\hat{c}_\phi$ is Hermitian and \emph{completely positive} if and only if $\hat{c}_\phi$ is positive semidefinite \cite{Szczygielski2023}.

\begin{proposition}\label{prop:MapHPCPcoCP}
Let a linear map $\phi : \matr{n}\to\matr{n}$ be covariant w.r.t.~the diagonal unitaries. Then:
\begin{enumerate}
	\item $\phi$ is \emph{Hermiticity preserving} if and only if $\hat{c}_\phi$ admits the form \eqref{eq:MatrixCgenStructure} with $C_1$ real, $C_2$ imaginary and $C_3$ Hermitian;
	\item $\phi\in\cpe{\matr{n}}$ if and only if $\hat{c}_\phi$ admits the form \eqref{eq:MatrixCgenStructure} with $C_1 \geqslant |C_2|$ and $C_3 \geqslant 0$;
	\item $\phi\in\cocpe{\matr{n}}$ if and only if $\hat{c}_{\transpose\circ\phi}$ admits the form \eqref{eq:MatrixCgenStructure2} with $(a_i)\in\reals_{+}^{n}$, $C_{11}, C_{22} \geqslant 0$, $C_{12} = C_{21}^{\hadj}$ and $|C_{12}| \geqslant \sqrt{C_{11}C_{22}}$.
\end{enumerate}
\end{proposition}

\begin{proof}
By Theorem \ref{thm:PhiUconjcovariant}, $\hat{c}_\phi$ is of a form \eqref{eq:MatrixCgenStructure} with diagonal $C_1$, $C_2$. Hence, $\hat{c}_\phi$ is Hermitian if and only if $C_1$, $C_3$ are Hermitian and $C_2$ is anti-Hermitian, i.e.~$C_1$ is real and $C_2$ is purely imaginary. For positive semidefiniteness, consider a characteristic polynomial
\begin{align}
	w(\lambda) &= \det{\left(\begin{array}{cc}C_1 - \lambda I & -C_2 \\ C_2 & C_1 - \lambda I\end{array}\right)} = \det{\left[(C_1 - \lambda I)^2 + C_{2}^{2}\right]} \\
	&= \det{\left(C_{1}^{2} - 2\lambda C_1 + \lambda^2 I + C_{2}^{2}\right)} = \det{\operatorname{diag}{\{s_{k}^{2} + a_{k}^{2} - 2\lambda s_k + \lambda^2\}}} \nonumber\\
	&= \prod_{k=1}^{n(n-1)/2} (s_{k}^{2} + a_{k}^{2} - 2\lambda s_k + \lambda^2),\nonumber
\end{align}
where we denoted $s_k \in \spec{C_1}$, $r_k \in \spec{C_2}$. Solving $w(\lambda) = 0$ we see
\begin{equation}
	\lambda_k = s_k \pm i r_k .
\end{equation}
where $s_k \geqslant 0$ and hence
\begin{equation}
	\spec{\hat{c}_\phi} = \spec{C_3} \cup \{s_k \pm ir_k : s_k \geqslant 0, r_k \in \spec{C_2}\}.
\end{equation}
Then we see $\lambda_k \geqslant 0$ is satisfied if and only if $r_k$ is imaginary, $r_k = i t_k$ for $t_k \in \reals$ such that $s_k \pm ir_k = s_k \pm t_k \geqslant 0$, or if $s_k \geqslant |r_k|$; hence $C_1 - |C_2|$ must be positive semidefinite, as stated. Finally, $\phi\in\cocpe{\matr{n}}$ is covariant if and only if $\phi = \transpose\circ\psi$ for $\psi\in\cpe{\matr{n}}$ conjugate covariant by Theorem \ref{thm:UcovUbarCov}. Theorem \ref{thm:PhiUconjcovariant} then shows $\hat{c}_\psi\geqslant 0$ must be positive semidefinite and of a form \eqref{eq:MatrixCgenStructure2}; this immediately yields $C_{11}, C_{22} \geqslant 0$, $C_{12} = C_{21}^{\hadj}$ and $C_{33} \geqslant 0$; via \eqref{eq:XmatrixDiagonal} then $a_i \geqslant 0$. For the remaining blocks, compute
\begin{align}
	w(\lambda) &= \det{\left(\begin{array}{cc}C_{11} - \lambda I & C_{21}^{\hadj} \\ C_{21} & C_{22} - \lambda I\end{array}\right)} \\
	&= \det{\left[(C_{11} - \lambda I)(C_{22}-\lambda I) - C_{21}^{\hadj}C_{21}\right]} \nonumber \\
	&= \det{\operatorname{diag}{\{(s_k-\lambda)(r_k-\lambda)-|w_k|^2\}}} \nonumber\\
	&= \prod_{k=1}^{n(n-1)/2}\left( s_k r_k - |w_k|^2 - (s_k + r_k)\lambda + \lambda^2 \right) \nonumber
\end{align}
for $s_k \in \spec{C_{11}}$, $r_k \in \spec{C_{22}}$ and $w_k \in \spec{C_{21}}$. Solving $w(\lambda) = 0$ we obtain
\begin{equation}
	\lambda_k = \frac{1}{2}\left( s_k + r_k \pm \sqrt{(s_k -r_k)^2 + 4|w_k|^2} \right).
\end{equation}
Since $r_k, s_k \geqslant 0$ we have $\lambda_k \geqslant 0$ if and only if $|w_k| \geqslant \sqrt{s_k r_k}$, or $|C_{12}|^2 \geqslant C_{11} C_{22}$.
\end{proof}

We conclude this section with the following corollary on decomposable covariant maps, which is an immediate consequence of Theorems \ref{thm:MainResultDec} and \ref{thm:UcovUbarCov}:

\begin{theorem}\label{thm:MaxCovDec}
A linear map $\varphi \in \dece{\matr{n}}$ is covariant w.r.t.~the maximal commutative subgroup of $\mathrm{U}(n)$ if and only if
\begin{equation}
	\varphi = \phi + \transpose\circ\psi
\end{equation}
for some $\phi\in\cpe{\matr{n}}$ covariant and $\psi\in\cocpe{\matr{n}}$ conjugate covariant, both characterized in Proposition \ref{prop:MapHPCPcoCP}.
\end{theorem}

\section{Decomposable quantum dynamics}
\label{sec:Ddiv}

In this section we highlight some applications of general theory of decomposable covariant maps, as elaborated earlier, in the description of quantum dynamics.

Let $B_1 (H)$ be the Banach space of trace class operators on Hilbert space $H$. A positive semidefinite operator $\rho\in B_1 (H)$ of trace $1$ will be called a \emph{density operator} (or \emph{density matrix} when $H\simeq \complexes^n$). Density operators are traditionally used to express (mixed) states of physical systems in quantum mechanics. Their time dependence is then characterized in terms of a \emph{quantum evolution family}, or \emph{quantum dynamical map}, i.e.~a time-parametrized family $(\Lambda_t)_{t\geqslant 0}$ of positive, trace preserving linear maps on $B_1 (H)$ such that $\rho_t = \Lambda_t (\rho_0)$ for some initial $\rho_0$. The most common approach further requires $\Lambda_t$ to be not only positive, but rather CP for all $t\geqslant 0$ \cite{Alicki2006,Breuer2002,Rivas2012}.

An evolution is called \emph{divisible} if and only if there exists a map $V_{t,s}$ on $B_1 (H)$ such that
\begin{equation}
	\Lambda_t = V_{t,s}\circ\Lambda_s \quad \text{for all } t\geqslant 0,\,s\in [0,t];
\end{equation}
$V_{t,s}$ is then called a \emph{propagator}. Further, evolution is \emph{P-divisible} when $V_{t,s}$ is positive and \emph{CP-divisible}, or \emph{Markovian}, when $V_{t,s}$ is CP, for all $t$, $s$. Recently, a new class of the so-called \emph{D-divisible} evolution families on $\matr{n}$, generalizing Markovian ones, was proposed in \cite{Szczygielski2023}. The defining property is that $(\Lambda_t)_{t\geqslant 0}$ is D-divisible (or decomposably divisible) if it is divisible and its propagator $V_{t,s}$ is \emph{decomposable} for all $t,s$. In particular it was shown that D-visibility of trace preserving families satisfying a time-local Master Equation is totally characterized in terms of an associated \emph{generator}. In subsection \ref{sec:DdivisibleFamilies} below we elaborate more on a D-divisible case and provide necessary and sufficient conditions for covariance of such families.

\subsection{Some general remarks}

Recall that $B(H)$ is (isometrically isomorphic to) the \emph{dual} of $B_1 (H)$, with the duality pairing given by the trace, i.e.~for any bounded linear functional $f \in B_{1}(H)^*$ there exists a unique $A \in B(H)$ such that $f(T) = \tr{AT}$ for each $T\in B_1 (H)$. Since $\mathscr{A}\simeq B(H_0)$ for some Hilbert space $H_0$, the predual $\mathscr{A}_*$ is isomorphic to $B_1 (H_0)$. Therefore, for any linear map $\Phi : B_1 (H) \to \mathscr{A}_* \simeq B_1 (H_0)$, its \emph{dual} map $\Phi^\adj : \mathscr{A}\simeq B(H_0) \to B(H)$ can be, with a slight abuse of terminology, defined by property $\tr{A \Phi(T)} = \tr{\Phi^\adj (A)T}$ for $T\in B_1 (H)$ and $A\in B(H_0)\simeq\mathscr{A}$ being (a representation of) a linear functional on $\mathscr{A}_*$.

Now, let $\phi : \mathscr{A}\to B(H)$ be $U$-covariant, so that
\begin{equation}
	\phi \circ g = \mathrm{Ad}_{U(g)} \circ \phi
\end{equation}
for $U : \mathcal{G}\to B(H)$ and all $g\in\mathcal{G}$, as in previous sections. Assume that a map $\phi_* : B_1 (H) \to \mathscr{A}_*$ is the predual of $\phi$. One easily shows

\begin{proposition}\label{prop:CovMapDual}
A linear map $\phi : \mathscr{A}\to B(H)$ is $U$-covariant if and only if its predual $\phi_*$ is covariant in the sense that
\begin{equation}
	\phi_*\circ\mathrm{Ad}_{U(g)}^{-1} = g_* \circ \phi_*
\end{equation}
for all $g\in\mathcal{G}$, where $g_* : \mathscr{A}_* \to \mathscr{A}_*$ is the predual of $g$.
\end{proposition}

Let $(\Lambda_t)_{t\geqslant 0}$ be norm-continuous quantum evolution family as introduced at the beginning of this section and assume it is subject to the Cauchy problem
\begin{equation}\label{eq:LambdaODE}
	\frac{d}{dt}\Lambda_t = L_t \circ \Lambda_t, \quad \Lambda_0 = \id{},
\end{equation}
i.e.~the \emph{Master Equation}, for a bounded generator $L_t$ such that $\tr{L_t (\rho)} = 0$ for all $\rho\in B_1 (H)$; the derivative is to be understood in norm topology sense. Its dual $\Phi_t : B(H) \to B(H)$ can be then quickly checked to satisfy a dual Cauchy problem
\begin{equation}\label{eq:PhiODE}
	\frac{d}{dt}\Phi_t = \Phi_t \circ \mathcal{L}_t, \quad \Phi_0 = \id{},
\end{equation}
for $\mathcal{L}_t$ a map dual to $L_t$. Family $(\Phi_t)_{t\geqslant 0}$ is sometimes referred to as the \emph{Heisenberg picture} of the evolution, while $(\Lambda_t)_{t\geqslant 0}$ is the \emph{Schr\"odinger picture} \cite{Alicki2006}. Since the general theory earlier was established for maps on C*-algebras, we will stay with the dual picture and consider rather map $\Phi_t$ than $\Lambda_t$.

Let us take $\mathscr{A}=B(H)$. The requirement for $\Phi_t$ to satisfy \eqref{eq:PhiODE} has an immediate consequence on a particular form of covariance:

\begin{proposition}
Let $(\Phi_t)_{t\geqslant 0}$, $\Phi_t : B(H)\to B(H)$, be subject to the ODE \eqref{eq:PhiODE}. If $\Phi_t$ is $U$-covariant then necessarily $g(A) = \mathrm{Ad}_{U(g)}(A)$ for all $A\in B(H)$.
\end{proposition}

\begin{proof}
Note that, by $U$-covariance property and boundedness of $\mathrm{Ad}_{U(g)}$,
\begin{align}\label{eq:PhiLCov}
	\Phi_t\circ \mathcal{L}_t\circ g &= \lim_{\epsilon \to 0^+} \frac{1}{\epsilon}\left( \Phi_{t+\epsilon}\circ g - \Phi_t \circ g\right) \\
	&= \mathrm{Ad}_{U(g)}\circ\lim_{\epsilon\to 0^+} \frac{1}{\epsilon} \left( \Phi_{t+\epsilon} - \Phi_t \right) \nonumber \\
	&= \mathrm{Ad}_{U(g)}\circ\Phi_t\circ \mathcal{L}_t,\nonumber
\end{align}
so $\Phi_t\circ \mathcal{L}_t$ is also $U$-covariant. Integrating both sides of ODE \eqref{eq:PhiODE} we obtain a linear Volterra equation
\begin{equation}
	\Phi_t = \id{} + \int_{0}^{t} \Phi_s \circ L_{s}^{\adj} \, ds .
\end{equation}
After composing it with $g$ and applying \eqref{eq:PhiLCov}
\begin{equation}\label{eq:PhiG1}
	\Phi_t \circ g = g + \mathrm{Ad}_{U(g)}\circ\int_{0}^{t} \Phi_s \circ L_{s}^{\adj} \, ds .
\end{equation}
On the other hand, by covariance we have
\begin{equation}\label{eq:PhiG2}
	\Phi_t \circ g = \mathrm{Ad}_{U(g)}\circ \Phi_t = \mathrm{Ad}_{U(g)}\circ \left( \id{} + \int_{0}^{t} \Phi_s \circ L_{s}^{\adj}\, ds \right).
\end{equation}
Equating \eqref{eq:PhiG1} and \eqref{eq:PhiG2} we see it must hold $g = \mathrm{Ad}_{U(g)}$.
\end{proof}

In consequence the only well-posed $U$-covariance condition for family $(\Phi_t)_{t\geqslant 0}$ must be of a form
\begin{equation}\label{eq:TheUcovariance}
	\Phi_t \circ \mathrm{Ad}_{U(g)} = \mathrm{Ad}_{U(g)}\circ\Phi_t ,
\end{equation}
satisfied for all $t\geqslant 0$, $g\in\mathcal{G}$; we will be only considering such covariance from now on until the end of this section. Since $g_* = \mathrm{Ad}^{-1}_{U(g)}$ by direct check, Proposition \ref{prop:CovMapDual} also shows that $\Phi_t$ is $U$-covariant if and only if its predual $\Lambda_t$ is $U^{-1}$-covariant, so all of the results can be easily translated to the dynamical maps itself by switching the representation $U$ to its inverse.

Recall that $(\Lambda_t)_{t\geqslant 0}$ is divisible if and only if $\Lambda_t = V_{t,s}\circ\Lambda_s$ for trace preserving propagator $V_{t,s} : B_1 (H) \to B_1 (H)$. The dual family $(\Phi_t)_{t\geqslant 0}$ is then also \emph{divisible}, i.e.
\begin{equation}\label{eq:DualDivisible}
	\Phi_t = \Phi_s \circ \mathcal{V}_{t,s}, \quad 0\leqslant s \leqslant t
\end{equation}
for $\mathcal{V}_{t,s}$ dual to $V_{t,s}$. Likewise, all forms of X-divisibility mentioned so far (for X = P, CP or D) apply to $(\Phi_t)_{t\geqslant 0}$ as well, i.e.~the dual family is, say, D-divisible if and only if $\mathcal{V}_{t,s}$ is decomposable, which happens if and only if $(\Lambda_t)_{t\geqslant 0}$ is itself D-divisible and so on.

\begin{proposition}\label{prop:Vcov}
Let $(\Phi_t)_{t\geqslant 0}$ be divisible. Then the following statements hold:
\begin{enumerate}
	\item $(\Phi_{t})_{t\geqslant 0}$ is $U$-covariant if and only if $\im{\comm{\mathcal{V}_{t,s}}{\mathrm{Ad}_{U(g)}}} \subseteq \ker{\Phi_s}$ for all $t\geqslant 0$, $s\leqslant t$;
	\item If $(\Phi_t)_{t\geqslant 0}$ is invertible then it is $U$-covariant if and only if $(\mathcal{V}_{t,s})_{s\leqslant t}$ is $U$-covariant.
\end{enumerate}
\end{proposition}

\begin{proof}
Divisibility condition \eqref{eq:DualDivisible} combined with \eqref{eq:TheUcovariance} implies
\begin{equation}
	\Phi_s \circ \mathcal{V}_{t,s} \circ \mathrm{Ad}_{U(g)} = \mathrm{Ad}_{U(g)} \circ \Phi_s \circ \mathcal{V}_{t,s} = \Phi_s \circ \mathrm{Ad}_{U(g)} \circ \mathcal{V}_{t,s},
\end{equation}
so for any $A\in B(H)$ we have
\begin{equation}\label{eq:VtsCommAd}
	\Phi_s \circ \left( \mathcal{V}_{t,s} \circ \mathrm{Ad}_{U(g)} - \mathrm{Ad}_{U(g)} \circ \mathcal{V}_{t,s} \right)(A) = 0,
\end{equation}
i.e.~$\comm{\mathcal{V}_{t,s}}{\mathrm{Ad}_{U(g)}}(A) \in \ker{\Phi_s}$ for all $s\leqslant t$, proving the first statement. For the second one, when $\Phi_t$ is invertible it has a trivial kernel so \eqref{eq:VtsCommAd} is satisfied if and only if $\ker{\comm{\mathcal{V}_{t,s}}{\mathrm{Ad}_{U(g)}}}$ is the whole $B(H)$, i.e.~when $\mathcal{V}_{t,s}$ and $\mathrm{Ad}_{U(g)}$ commute; this however is the $U$-covariance condition for $\mathcal{V}_{t,s}$.
\end{proof}

An analogous result regarding the generator itself can also be formulated (we omit the proof as it goes along the same lines):

\begin{proposition}\label{prop:PhiCovInv}
Let $(\Phi_t)_{t\geqslant 0}$ satisfy ODE \eqref{eq:PhiODE}. The following statements hold:
\begin{enumerate}
	\item If $\Phi_t$ is $U$-covariant then $\im{\comm{\mathcal{L}_t}{\mathrm{Ad}_{U(g)}}}\subseteq\ker{\Phi_t}$;
	\item If $\Phi_t$ is invertible then it is $U$-covariant if and only if $\mathcal{L}_t$ is $U$-covariant.
\end{enumerate}
\end{proposition}

Naively speaking, covariance of the generator seems as the most natural and intuitive measure of covariance of the dynamics, despite being not a necessary condition but only a sufficient one. However, a very frequently appearing scenario in applications is that the generator which defines equation of motion for density operators (or equivalently observables in Heisenberg picture) is \emph{regular}, i.e.~$t\mapsto\| L_t \|$ is uniformly bounded over (sufficiently large interval in) $\reals_+$. In such case the dynamical maps are always invertible and so the $U$-covariance of the generator and dynamics itself become equivalent.

\begin{proposition}\label{prop:ODEregularUcovL}
Let the quantum evolution family $(\Lambda_t)_{t\geqslant 0}$ satisfy ODE \eqref{eq:LambdaODE} with $L_t$ regular. Then, $(\Lambda_t)_{t\geqslant 0}$ is U-covariant if and only if $L_t$ is U-covariant for all $t\geqslant 0$.
\end{proposition}

\begin{proof}
By Proposition \ref{prop:CovMapDual} the $U$-covariance of $\Lambda_t$ is equivalent to $U^{-1}$-covariance of $\Phi_t$. Regularity of $L_t$ then implies regularity of $\mathcal{L}_t$ and invertibility of $\Phi_t$ in consequence, so Proposition \ref{prop:PhiCovInv} invoked for representation $U^{-1}$ shows $\mathcal{L}_t$ must be $U^{-1}$-covariant, i.e.~$L_t$ is $U$-covariant.
\end{proof}

\subsection{D-divisible families on \texorpdfstring{$\matr{n}$}{matrix algebras}}
\label{sec:DdivisibleFamilies}

In \cite{Szczygielski2023} it was shown that the quantum evolution family $(\Lambda_t)_{t\geqslant 0}$ on $\matr{n}$ which satisfies an ODE \eqref{eq:LambdaODE} with $L_t$ regular is D-divisible, i.e.~its propagator $V_{t,s}$ is a \emph{decomposable} map, if and only if the generator admits a structure
\begin{equation}\label{eq:LtDdivisible}
	L_t = -i\comm{H_t}{\cdot} + \varphi_t - \frac{1}{2}\acomm{\varphi_{t}^{\adj}(I)}{\cdot}
\end{equation}
for Hermitian $H_t$ and decomposable map $\varphi_t$. In time-independent case equation \eqref{eq:LtDdivisible} defines the most general form of the generator of trace preserving (or unital, in dual version) decomposable semigroup on $\matr{n}$, which extends seminal result by V.~Gorini, A.~Kossakowski, E.~C.~G.~Sudarshan \cite{Gorini1976} and G.~Lindblad \cite{Lindblad1976} for completely positive semigroups. We remark here that characterization \eqref{eq:LtDdivisible} applies only to finite-dimensional C*-algebras and it is currently unknown if it admits a direct generalization to the case of any $B(H)$. Such characterization, if true, would provide a substantial extension of \cite{Lindblad1976} and would contribute to more general theory of dynamical semigroups beyond complete positivity.

For start, let us take the generator constant, $L_t = L$, so the evolution family is simply a semigroup, $(e^{tL})_{t\geqslant 0}$. The following theorem can be considered a generalization of result by Holevo \cite{Holevo1993} towards decomposable case, applied for finite-dimensional C*-algebras:

\begin{proposition}\label{prop:LcovDec}
Generator $L$ of decomposable and trace preserving semigroup on $\matr{n}$ is $U$-covariant if and only if it admits a form \eqref{eq:LtDdivisible} (with time dependence suppressed) for Hermitian matrix $H\in U(\mathcal{G})^\prime$ and a $U$-covariant map $\varphi\in\dece{\matr{n}}$.
\end{proposition}

\begin{proof}
Only the ``$\Rightarrow$'' direction needs to be shown. Let $L$ be of a form \eqref{eq:LtDdivisible} for Hermitian $H$ and decomposable $\phi$ and assume it is $U$-covariant. Define yet again the projection $\proj{U} : B(\matr{n})\to\mathcal{C}_U$ onto subspace $\mathcal{C}_U$ of $U$-covariant maps,
\begin{equation}
	\proj{U}(\phi) = \int_\mathcal{G} \mathrm{Ad}_{U(g)} \circ \phi \circ \mathrm{Ad}_{U(g)}^{-1} \, d\mu(g),
\end{equation}
i.e.~as in the proof of Theorem \ref{thm:MainResultDec}. Naturally, we have $\proj{U}(L) = L$ and, after some easy algebra,
\begin{equation}
	\proj{U}(L)(a) = -i \comm{\hat{H}}{a} + \psi(a) - \frac{1}{2}\acomm{\psi^{\prime}(I)}{a}
\end{equation}
for
\begin{equation}
	\hat{H} = \int_\mathcal{G} U(g) H U(g)^\hadj \, d\mu(g), \quad \psi = \proj{U}(\phi).
\end{equation}
Map $\psi$ is then $U$-covariant and still decomposable by Theorem \ref{thm:MainResultDec}. Moreover, for any $g\in\mathcal{G}$ we have
\begin{align}
	U(g) \hat{H} &= \int_\mathcal{G} U(g)U(h) H U(h)^\hadj \, d\mu(h) = \int_\mathcal{G} U(g) U(g^{-1}h) H U(g^{-1}h)^{\hadj} \, d\mu(h)\\
	&= \int_\mathcal{G} U(h) H U(h) U(g)\, d\mu(h) = \hat{H} U(g),\nonumber
\end{align}
by translation invariance of the integral, i.e.~$U(g)$ commutes with $\hat{H}$ for all $g\in\mathcal{G}$ and the proof is finished.
\end{proof}

We conclude this section with the following

\begin{theorem}\label{thm:SemigroupStructureTheorem}
Let $U : \mathcal{G}\to\matr{n}$ be a unitary representation. A semigroup $(e^{tL})_{t\geqslant 0}$ of decomposable trace preserving maps on $\matr{n}$ is $U$-covariant if and only if there exists a Hermitian matrix $H\in U(\mathcal{G})^\prime$ and a $U$-covariant map $\varphi\in\dece{\matr{n}}$ such that $L$ is of a form \eqref{eq:LtDdivisible} (with time suppressed).
\end{theorem}

\begin{proof}
Since all maps of a form $e^{tL}$ are invertible, trivial application of Propositions \ref{prop:ODEregularUcovL} and \ref{prop:LcovDec} yields the claim. 
\end{proof}

As a final result of the article, we present a generalization of Theorem \ref{thm:SemigroupStructureTheorem} which covers the case of regular time-dependent generators.

\begin{theorem}\label{thm:DdivisibleStructureTheorem}
Let $(\Lambda_t)_{t\geqslant 0}$ be a D-divisible family of trace preserving maps on $\matr{n}$ satisfying ODE \eqref{eq:LambdaODE} with $L_t$ regular and let $U : \mathcal{G}\to\matr{n}$ be a unitary representation. Then, $(\Lambda_t)_{t\geqslant 0}$ is $U$-covariant if and only if there exist a Hermitian matrix $H_t\in U(\mathcal{G})^\prime$ and a $U$-covariant map $\varphi_t \in \dece{\matr{n}}$ such that $L_t$ admits a form \eqref{eq:LtDdivisible}.
\end{theorem}

\begin{proof}
When the family is D-divisible, \cite{Szczygielski2023} implies $L_t$ must be of a form \eqref{eq:LtDdivisible} for Hermitian $H_t$ and decomposable $\varphi_t$. Regularity of $L_t$ implies, via Proposition \ref{prop:ODEregularUcovL}, that $(\Lambda_t)_{t\geqslant 0}$ is $U$-covariant if and only if $L_t$ is $U$-covariant for all $t\geqslant 0$. From Proposition \ref{prop:LcovDec} it then follows that $H_t \in U(\mathcal{G})^\prime$ and $\varphi_t$ is $U$-covariant. Conversely, when $L_t$ admits the claimed form then it is automatically $U$-covariant, so Proposition \ref{prop:ODEregularUcovL} yields the family $(\Lambda_t)_{t\geqslant 0}$ must be $U$-covariant, and D-divisible by \cite{Szczygielski2023}.
\end{proof}

\section{Summary}
\label{sec:Summary}

We presented some extensions of a well-established theory of covariant completely positive maps on C*-algebras onto the broader class of decomposable maps. Analogously to the initial seminal work of \cite{Scutaru1979}, a characterization of such was given in a form of a dilation theorem. A possible quantum physics application was highlighted by providing necessary and sufficient conditions for covariance of decomposable quantum dynamical maps admitting regular, time-dependent generators. It would be possibly of interest to verify applicability of Theorems \ref{thm:SemigroupStructureTheorem} and \ref{thm:DdivisibleStructureTheorem} to maps on a general C*-algebra $B(H)$. This however is presently unclear due to the lack of characterization of decomposable, norm-continuous contraction semigroups in infinite dimension in the spirit of Lindblad \cite{Lindblad1976}, as we remarked earlier.

\appendix
\section{Mathematical supplement}
\label{app:SideTheorems}

\begin{theorem}\label{thm:CuNormClosed}
Let $\mathcal{C}_U$ denote the set of all $U$-covariant bounded linear maps from $\mathscr{A}$ to $B(H)$. Then, $\mathcal{C}_U$ is a closed linear subspace in $B(\mathscr{A},B(H))$.
\end{theorem}

\begin{proof}
$\mathcal{C}_U$ is clearly a subspace since for any two $U$-covariant maps $\phi_1$, $\phi_2$ their linear combination is also $U$-covariant. Let then $\phi\in B(\mathscr{A},B(H))$ be a limit point of $\mathcal{C}_U$, i.e.~let there exist a sequence $(\phi_n)$, $\phi_n \neq \phi$, of $U$-covariant maps such that $\phi_n \to \phi$ in norm. Mapping $\phi\mapsto\mathrm{Ad}_U \circ\phi$ is an isometry in $B(\mathscr{A},B(H))$ for any unitary $U$, so we have
\begin{equation}
	\| \phi_n - \phi \| = \| \mathrm{Ad}_{U(g)}\circ (\phi_n - \phi) \| = \| \phi_n \circ g - \mathrm{Ad}_{U(g)}\circ\phi \|
\end{equation}
for all $g\in\mathcal{G}$, and $(\phi_n \circ g)$ converges to $\mathrm{Ad}_{U(g)}\circ\phi$ in norm. On the other hand,
\begin{equation}
	\| \phi_n \circ g - \phi\circ g\| \leqslant \| \phi_n - \phi\| \cdot \|g(\cdot)\| \to 0
\end{equation}
hence $(\phi_n \circ g)$ converges to $\phi\circ g$. This implies $\phi\circ g = \mathrm{Ad}_{U(g)}\circ\phi$ by uniqueness of a limit, i.e.~$\phi\in\mathcal{C}_U$ and subspace $\mathcal{C}_U$ is closed in norm topology.
\end{proof}

\begin{theorem}\label{prop:ConesInvarianceP}
Let $C$ denote either $\cp{\mathscr{A}}{H}$ or $\cocp{\mathscr{A}}{H}$. Then we have $\proj{U}(C) \subset C$.
\end{theorem}

\begin{proof}Let $\{\psi_g\}_{g\in\mathcal{G}}\subset\cp{\mathscr{A}}{H}$ be a family of CP maps s.t.~$g\mapsto \psi_g$ is measurable and uniformly bounded, i.e.~$\sup_{g\in\mathcal{G}}\|\psi_g\| < \infty$. Then, by Stinespring's dilation theorem \cite{Stinespring_1955} there exists a family $\{K_g\}_{g\in\mathcal{G}}$ of Hilbert spaces, family $\{V_g\}_{g\in\mathcal{G}}$ of bounded linear operators $V_g : H\to K_g$ and a family $\{\pi_g\}_{g\in\mathcal{G}}$ of *-homomorphisms $\pi_g : \mathscr{A}\to B(K_g)$ such that $\psi_g (a) = V_{g}^{\hadj} \pi_g (a) V_{g}$ for every $a\in\mathscr{A}$, $g\in\mathcal{G}$. Define a Hilbert space
\begin{equation}
	K = \int_{\mathcal{G}}^{\oplus} K_g \, d\mu(g),
\end{equation}
i.e.~a direct integral of family $\{K_g\}_{g\in\mathcal{G}}$, as well as linear operator $V : H\to K$ by
\begin{equation}
	Vx = \int_{\mathcal{G}}^{\oplus} V_g x \, d\mu(g), \quad x\in H,
\end{equation}
and a *-automorphism $\pi : \mathscr{A}\to B(K)$ by
\begin{equation}
	\pi (a) = \int_{\mathcal{G}}^{\oplus} \pi_g (a) \, d\mu(g), \quad a\in\mathscr{A}.
\end{equation}
Let us denote inner products in spaces $H$, $K_g$ and $K$ by $\iprod{\cdot}{\cdot}_H$, $\iprod{\cdot}{\cdot}_g$ and $\iprod{\cdot}{\cdot}_K$, respectively. Then, for any $x,y\in H$, $a\in\mathscr{A}$ we have
\begin{align}
	\iprod{x}{V^\hadj \pi(a) Vy}_H &= \iprod{Vx}{\pi(a)Vy}_K = \int_{\mathcal{G}} \iprod{V_g x}{\pi_{g}(a)V_g y}_g \, d\mu(g) \\
	&= \int_\mathcal{G} \iprod{x}{V_{g}^{\hadj}\pi_{g}(a)V_g y}_H \, d\mu(g) = \iprod{x}{\int_\mathcal{G} V_{g}^{\hadj}\pi_{g}(a)V_g \, d\mu(g) \, y}_H \nonumber \\
	&= \iprod{x}{\int_\mathcal{G} \psi_g (a) \, d\mu(g) \, y}_H \nonumber
\end{align}
which yields equality of operators
\begin{equation}
	V^\hadj \pi(a) V = \int_\mathcal{G} \psi_g (a) \, d\mu(g)
\end{equation}
satisfied for all $a\in\mathscr{A}$. This means the integral $\int_\mathcal{G} \psi_g \, d\mu(g)$ of a completely positive family $\{\psi_g\}_{g\in\mathcal{G}}$ admits a Stinespring form, i.e.~is also CP. Now, putting
\begin{equation}
	\psi_g (a) = \mathrm{Ad}_{U(g)}\circ\phi\circ g^{-1}
\end{equation}
where $\phi$ is CP yields $\proj{U}(\phi) \in \cp{\mathscr{A}}{H}$. The remaining coCP case then follows along the same lines after replacing a *-homomorphism with a *-antihomomorphism in Stinespring's form.
\end{proof}

\begin{theorem}\label{thm:CommFunctionsOperators}
Let $A,B\in B(X)$ for a Banach space $X$ and let $f,g : \complexes\to\complexes$. If there exist open sets $U_A$, $U_B$ containing spectra of $A$ and $B$, respectively, such that $f : U_A \to \complexes$ and $g : U_B \to \complexes$ are holomorphic and injective, then $\comm{f(A)}{g(B)} = 0$ if and only if $\comm{A}{B}=0$.
\end{theorem}

\begin{proof}
The proof method is inspired by the one employed by Wermuth in \cite{Wermuth1997} where a special case of $f$, $g$ being operator exponentials was considered. We note that a virtually identical result, albeit achieved with different methods, was quite recently shown by Paliogiannis \cite{Paliogiannis2023}.

Note that since $f$ and $g$ are assumed holomorphic, power series expansions of $f(A)$ and $g(B)$ commute term by term so only ``$\Rightarrow$'' direction needs to be addressed. Let then $\spec{A} \subset U_A$, $\spec{B} \subset U_B$ with $U_A$, $U_B$ open and bounded (we may assume so since both $\spec{A}$, $\spec{B}$ are compact). By open mapping theorem, both images $f(U_A)$, $g(U_B)$ are open subsets of $\complexes$, while continuity implies that $f(\spec{A})$ and $f(\spec{B})$ are also compact. By injectivity assumption, both $f$ and $g$ are then invertible, i.e.~there exist functions
\begin{equation}
f^{-1} : f(U_A) \to U_A, \quad g^{-1} : g(U_B) \to U_B,
\end{equation}
which are also necessarily holomorphic on $f(U_A)$ and $f(U_B)$ \cite[Theorem 10.33]{Rudin2013}. Since $f(\spec{A})$ is closed and embedded in an open set $f(U_A)$, it may be covered by another open set $D_{A,\epsilon}$ defined as a union of all open balls, each one of radius $\epsilon$, covering all possible points of $f(\spec{A})$,
\begin{equation}
D_{A,\epsilon} = \bigcup_{z\in f(\spec{A})} K_{z,\epsilon}
\end{equation}
where such $\epsilon > 0$ is taken that $f(\spec{A}) \subset D_{A,\epsilon} \subset f(U_A)$. Then, compactness of $f(\spec{A})$ infers existence of a finite subcover $E_{A,\epsilon}$ of $f(\spec{A})$,
\begin{equation}
E_{A,\epsilon} = \bigcup_{j=1}^{n} K_{z_j, \epsilon}
\end{equation}
where $\{z_j : 1 \leqslant j \leqslant n\}\subset f(\spec{A})$, such that $f(\spec{A}) \subset E_{A,\epsilon} \subset D_{A.\epsilon}$. Then, $\overline{E_{A,\epsilon}}$ is compact and its boundary $\gamma_A$ is a Jordan curve, i.e.~closed, continuous and with no self-intersections. By holomorphic functional calculus, operator $A = (f^{-1}\circ f)(A)$ can be expressed as an integral
\begin{equation}
	A = \frac{1}{2\pi i}\oint_{\gamma_A}f^{-1}(z)\left(z-f(A)\right)^{-1}dz.
\end{equation}
By Runge's theorem \cite[Theorem 13.9]{Rudin2013}, a holomorphic function $f^{-1}$ may be approximated by a sequence $(r_n)$ of rational functions, $r_n (z) = p_n (z)/q_n(z)$ for $p_n$, $q_n$ polynomials, converging to $f^{-1}$ uniformly on $\overline{E_{A,\epsilon}}$ and it is always possible to choose polynomials $q_n$ that have no zeros in $\overline{E_{A,\epsilon}}$. Functions $r_n$ are therefore holomorphic in this set and in particular, this implies that the operator
\begin{equation}
	q_n (f(A))^{-1} = \frac{1}{2\pi i}\oint_{\gamma_A}\frac{1}{q_n (z)}\left(z-f(A)\right)^{-1}dz
\end{equation}
exists and is bounded, i.e.~$r_n (f(A)) = p_n (f(A))q_{n}(f(A))^{-1}$ is well-defined and $\| r_n (f(A)) - A\|\to 0$ as $n\to\infty$. The same reasoning may be then carried out for $B$ and function $g$, resulting in existence of a sequence $(s_n)$ of rational functions, with poles outside a compact set containing $g(\spec{B})$, and converging to $g^{-1}$ uniformly on this set; in result, a sequence $(s_n (g(B)))$ exists in $B(X)$ and converges to $B$ in norm. From this it is easy to see that
\begin{equation}
	\Delta_n = r_n(f(A))s_n(g(B)) - s_n(g(B))r_n(f(A)) \to \comm{A}{B}
\end{equation}
in norm when $n\to\infty$. However, when it is assumed that $f(A)$ and $g(B)$ commute then also all rational functions of these operators commute, i.e.~$\Delta_n = 0$ and the result follows. 
\end{proof}

\begin{theorem}\label{thm:CommFunctionsOperatorsExp}
Let $A,B \in B(X)$ for a Banach space $X$. If spectrum of $A$ is $2\pi i$--congruence free, i.e.~if no $\lambda_1$, $\lambda_2$ exist in $\spec{A}$ such that $\lambda_1 -\lambda_2 = 2k\pi i$ for some $k\in\integers\setminus\{0\}$, then $\comm{e^A}{B} = 0$ if and only if $\comm{A}{B}=0$.
\end{theorem}

\begin{proof}
When $\sigma(A)$ is $2\pi i$--congruence free, the function $f(z) = e^z$ is injective on an open cover
\begin{equation}
	D_\epsilon = \bigcup_{z\in\spec{A}}K_{z,\epsilon}
\end{equation}
of $\spec{A}$, for $\epsilon > 0$ small enough \cite{Wermuth1997}. Hence, the claim follows from Theorem \ref{thm:CommFunctionsOperators}.
\end{proof}

Below we supply the Reader with an alternative proof which employs Wermuth's result in a more direct manner by invoking a simple observation of ``shrinking'', or ``rescaling'' property of bounded subsets of complex plane, stated as a Lemma \ref{lemma:z0congruenceScaling}:

\begin{lemma}\label{lemma:z0congruenceScaling}
For every bounded nonempty set $U\subset\mathbb{C}$ and every $z \in\mathbb{C} \setminus \{0\}$ there exists such $\tau > 0$ small enough, that set $t U = \{t w : w\in U \}$ is $z$--congruence free for every $t \in (0,\tau)$.
\end{lemma}

\begin{proof}
Let $U \subset \mathbb{C}$ be nonempty and bounded and let
\begin{equation}
	\Delta = \sup_{z_1, z_2 \in U}{|z_1 - z_2|}
\end{equation}
be its \emph{diameter}. If $\Delta > 0$, define $\tau = \frac{|z|}{\Delta}$. Then, for every $t \in (0, \tau)$ and every pair $z_1, z_2 \in U$ we have $|z_1-z_2| \leqslant \Delta$ which yields
\begin{equation}\label{eq:tZdistance}
	t |z_1-z_2| < \tau \Delta = |z|.
\end{equation}
Let $tU = \{tz : z \in U\}$ and take any $w_1, w_2 \in tU$. As $w_1 = tz_1$ and $w_2 = tz_2$ for some $z_1, z_2 \in U$, equation \eqref{eq:tZdistance} implies 
\begin{equation}
	|w_{1} - w_{2}| = t | z_1 - z_2| < |z|,
\end{equation}
which automatically results in
\begin{equation}
	w_{1} - w_{2} \neq k z, \qquad k \in \mathbb{Z}\setminus \{0 \}
\end{equation}
for every $w_{1}, w_{2} \in t U$, i.e. $w_1 \neq w_2  \, (\mathrm{mod} \, z)$ and $t U$ is $z$--congruence free for any $t \in (0,\tau)$. On the other hand, if $\Delta = 0$, i.e. $U=\{ z_0 \}$ is a singleton, then set $tU$ is automatically $z$--congruence free by definition; then, one can take any $\tau \in (0, \infty )$.
\end{proof}

\begin{proof}[Proof of Theorem~\ref{thm:CommFunctionsOperatorsExp} (alternative)]
Assume $\comm{e^A}{B} = 0$. Then, $e^A$ commutes also with every analytic function of $B$, so in particular $\comm{e^A}{e^{tB}}=0$ for all $t\in\mathbb{R}$. As $\spec{B}$ is a nonempty bounded subset of $\mathbb{C}$, lemma \ref{lemma:z0congruenceScaling} invoked for $U = \spec{B}$ guarantees existence of such $\tau > 0$ that $t \spec{B}$ is $2\pi i$--congruence free for any $t\in(0,\tau )$. Therefore, $tB$ is of $2\pi i$--congruence free spectrum and by virtue of Wermuth's result, $\comm{e^A}{e^{tB}} = 0$ if and only if $\comm{A}{tB} = 0$ for every $t\in (0,\tau)$, so $A$ and $B$ must commute.
\end{proof}

\begin{proposition}\label{prop:AlphaStructure}
For each $g\in\mathcal{G}$,
\begin{enumerate}
	\item matrix $\hat{\alpha}_U(g)$ is unitary and admits a block structure
	\begin{equation}\label{eq:AlphaStructure}
		\hat{\alpha}_U(g) = \left( \begin{array}{ccc} K_{1}(g) & -K_{2}(g) & 0 \\ K_{2}(g) & K_{1}(g) & 0 \\ 0 & 0 & I_n \end{array}\right),
	\end{equation}
	where $I_n \in \matr{n}$ is an identity matrix and $K_1 (g), K_2(g) \in \mathrm{Diag}_{n(n-1)/2}(\complexes)$;
	\item matrix $\hat{\beta}_U(g)$ is unitary and admits a block diagonal structure
	\begin{equation}\label{eq:BetaStructure}
		\hat{\beta}_U(g) = \left( \begin{array}{ccc} R_{1}(g) & 0 & 0 \\ 0 & R_{1}(g) & 0 \\ 0 & 0 & R_2(g) \end{array}\right),
	\end{equation}
	where $R_1(g) \in \mathrm{Diag}_{n(n-1)}(\complexes)$ and $R_2 (g) \in \matr{n}$ is unitary and symmetric.
\end{enumerate}

\end{proposition}

\begin{proof}
Due to mutual orthogonality of subspaces $\mathrm{Sym}_{n}(\complexes)$, $\mathrm{Asym}_{n}(\complexes)$ and $\mathrm{Diag}_{n}(\complexes)$, matrix $\hat{\alpha}_U(g)$ will admit a natural block structure, $\hat{\alpha}_U(g) = [M_{ij}]_{i,j=1}^{3}$, where every block $M_{ij}$ corresponds to a certain pair of subspaces: for example, $M_{11} = [\hsiprod{F_{i}^{\mathrm{s}}}{{F_{j}^{\mathrm{s}}}^{\prime}}]$, $M_{12} = [\hsiprod{F_{i}^{\mathrm{s}}}{{F_{i}^{\mathrm{a}}}^{\prime}}]$ and so on. Matrices $F_{i}^{\mathrm{s}}$ and $F_{i}^{\mathrm{a}}$ come in pairs labeled by the same index $i$ running from $1$ up to $\frac{1}{2}n(n-1)$, which is uniquely determined by the same pair $(\mu,\nu)$ of indices appearing in \eqref{eq:FiSymAsym}. One easily checks that map $\mathrm{Ad}_{U(g)}^{-1}$ leaves $F_{i}^{\mathrm{d}}$ unchanged, and so blocks $M_{31}$, $M_{32}$, $M_{13}$ and $M_{23}$ vanish by orthogonality of subspaces and $M_{33} = I_n$. The remaining non-zero blocks $M_{ij}$ are determined by simply calculating appropriate inner products. With simple algebra involving property $E_{ij}E_{kl} = \delta_{jk} E_{il}$ one verifies
\begin{subequations}
	\begin{equation}
		\mathrm{Ad}_{U(g)}^{-1}(F_{i}^{\mathrm{s}}) = \cos{(x_\mu - x_\nu)} F_{i}^{\mathrm{s}} + \sin{(x_\mu - x_\nu)} F_{i}^{\mathrm{a}},
	\end{equation}
	\begin{equation}
		\mathrm{Ad}_{U(g)}^{-1}(F_{i}^{\mathrm{a}}) = -\sin{(x_\mu - x_\nu)} F_{i}^{\mathrm{s}} + \cos{(x_\mu - x_\nu)} F_{i}^{\mathrm{a}},
	\end{equation}
\end{subequations}
i.e.~$\mathrm{Ad}_{U(g)}^{-1}$ is effectively the rotation by angle $x_\mu - x_\nu$. This naturally leads to
\begin{subequations}
\begin{equation}
		\hsiprod{F_{i}^{\mathrm{s}}}{{F_{j}^{\mathrm{s}}}^{\prime}} = \hsiprod{F_{i}^{\mathrm{a}}}{{F_{j}^{\mathrm{a}}}^{\prime}} = \cos{(x_\mu - x_\nu)}\delta_{ij},
\end{equation}
\begin{equation}
	\hsiprod{F_{i}^{\mathrm{s}}}{{F_{j}^{\mathrm{a}}}^{\prime}} = -\hsiprod{F_{i}^{\mathrm{a}}}{{F_{j}^{\mathrm{s}}}^{\prime}} -\sin{(x_\mu - x_\nu)} \delta_{ij},
\end{equation}
\end{subequations}
here with $j$ uniquely determining $(\mu,\nu)$. Putting this together we have
\begin{equation}
	M_{11} = M_{22} = K_{1}(g) = \cos{\Delta}, \quad M_{21} = -M_{12} = K_2 (g) = \sin{\Delta}
\end{equation}
where $\Delta = \operatorname{diag}{\{x_{\mu(k)}-x_{\nu(k)}\}}$, $1\leqslant k \leqslant \frac{1}{2}n(n-1)$; here, notation $\mu(k)$, $\nu(k)$ indicates the bijection $k\mapsto (\mu,\nu)$. This indeed shows that $\hat{\alpha}_U(g)$ has the claimed structure \eqref{eq:AlphaStructure}.

For $\hat{\beta}_U(g) = [\hsiprod{F_i}{F_{j}^{\prime}}]$, $F_{j}^{\prime} = U(g)F_j U(g)$, let us introduce a similar block structure $\hat{\beta}_U(g) = [M_{ij}]$, with blocks $M_{ij}$ of the same size as in $\hat{\alpha}_U(g)$. Note that a mapping $A\mapsto U(g) A U(g)$ does not alter symmetry of $A$ since $U(g)$ is diagonal, i.e.~it has $\operatorname{Sym}_{n}{(\complexes)}$ and $\operatorname{Asym}_{n}{(\complexes)}$ as invariant subspaces. This implies that when $F_i$ and $F_j$ are of different symmetry, we have $\hsiprod{F_i}{U(g)F_j U(g)} = 0$ and immediately $M_{ij}$ vanishes when $i\neq j$. When $F_i$ and $F_j$ are chosen both symmetric or antisymmetric, one confirms with direct computation that
\begin{equation}
	\hsiprod{F_j}{F_{k}^{\prime}} = e^{i(x_{\mu(j)}+x_{\nu(j)})}\delta_{jk},
\end{equation}
i.e.~blocks $M_{11}$ and $M_{22}$ are diagonal, equal and populated with phase factors of a form $e^{i(x_{\mu(j)}+x_{\nu(j)})}$, $1\leqslant j \leqslant \frac{1}{2}n(n-1)$; hence, they are unitary. When $F_i$, $F_j$ are both diagonal, we have
\begin{equation}
	\hsiprod{F_{i}^{\mathrm{d}}}{{F_{i}^{\mathrm{d}}}^{\prime}} = \tr{F_{i}^{\mathrm{d}}U(g)F_{j}^{\mathrm{d}}U(g)} = \tr{F_{j}^{\mathrm{d}} U(g) F_{i}^{\mathrm{d}} U(g)} = \hsiprod{F_{j}^{\mathrm{d}}}{{F_{i}^{\mathrm{d}}}^{\prime}}
\end{equation}
by Hermiticity of Frobenius basis and cyclic property of trace, and so the remaining block $M_{33}$ is symmetric. Unitarity then follows from simple computation,
\begin{align}
	(M_{33}^{\hadj}M_{33})_{ij} &= \sum_{k=1}^{n} \overline{\hsiprod{F_k}{F_{i}^{\prime}}} \hsiprod{F_k}{F_{j}^{\prime}} = \hsiprod{\sum_{k=1}^{n}\hsiprod{F_k}{F_{i}^{\prime}}F_k}{F_{j}^{\prime}} \\
	&= \hsiprod{\proj{\mathrm{d.}}(F_{i}^{\prime})}{F_{j}^{\prime}} = \hsiprod{F_i}{F_j} = \delta_{ij},\nonumber
\end{align}
since $\proj{\mathrm{d.}}$, the projection onto $\mathrm{Diag}_{n}(\complexes)$, acts trivially on diagonal matrices. In result, $\hat{\beta}_U(g)$ is unitary and of the claimed structure \eqref{eq:BetaStructure}.
\end{proof}

\bibliographystyle{unsrt} 
\bibliography{CovariantDecomposableMapsBib}

@Article{Wermuth1997,
  author   = {Wermuth, E. M. E.},
  title    = {{A remark on commuting operator exponentials}},
  journal  = {Proc. Amer. Math. Soc.},
  year     = {1997},
  subtitle = {A CONCURRENT KEY EXCHANGE PROTOCOL BASED ON COMMUTING MATRICES},
  volume   = {125},
  number   = {6},
  pages    = {1685-1688},
  issn     = {0002-9939},
  doi      = {10.1090/S0002-9939-97-03643-5},
}

@Article{Paliogiannis2023,
  author  = {Paliogiannis, F. C.},
  journal = {Complex Anal. Oper. Th.},
  title   = {{Remarks on Commuting Functions of Operators}},
  year    = {2023},
  number  = {6},
  volume  = {17},
  doi     = {10.1007/s11785-023-01388-y},
}

@Article{Hioe1981,
  author       = {Hioe, F. T. and Eberly, J. H.},
  journaltitle = {Phys. Rev. Lett.},
  title        = {N-Level Coherence Vector and Higher Conservation Laws in Quantum Optics and Quantum Mechanics},
  doi          = {10.1103/physrevlett.47.838},
  number       = {12},
  pages        = {838--841},
  volume       = {47},
  journal      = {Phys. Rev. Lett.},
  year         = {1981},
}

@Article{Kimura2003,
  author       = {Kimura, G.},
  journaltitle = {Phys. Lett. A},
  title        = {{The Bloch vector for N-level systems}},
  doi          = {10.1016/s0375-9601(03)00941-1},
  number       = {5-6},
  pages        = {339--349},
  volume       = {314},
  journal      = {Phys. Lett. A},
  publisher    = {Elsevier {BV}},
  year         = {2003},
}

@Article{Scutaru1979,
  author  = {Scutaru, H.},
  journal = {Rep. Math. Phys.},
  title   = {Some remarks on covariant completely positive linear maps on {C*}-algebras},
  year    = {1979},
  number  = {1},
  pages   = {79--87},
  volume  = {16},
  doi     = {10.1016/0034-4877(79)90040-5},
}

@Article{Stinespring_1955,
  author  = {Stinespring, W. F.},
  journal = {Proc. Amer. Math. Soc.},
  title   = {Positive functions on {C*}-algebras},
  year    = {1955},
  number  = {2},
  pages   = {211--216},
  volume  = {6},
  doi     = {10.1090/s0002-9939-1955-0069403-4},
}

@Article{Stormer_1982,
  author  = {Størmer, E.},
  journal = {Proc. Amer. Math. Soc.},
  title   = {Decomposable positive maps on {C*}-algebras},
  year    = {1982},
  number  = {3},
  pages   = {402--404},
  volume  = {86},
  doi     = {10.1090/s0002-9939-1982-0671203-5},
}

@Book{Stoermer2013,
  author    = {Størmer, E.},
  publisher = {Springer Berlin Heidelberg},
  title     = {Positive Linear Maps of Operator Algebras},
  year      = {2013},
  isbn      = {9783642343698},
  doi       = {10.1007/978-3-642-34369-8},
  issn      = {1439-7382},
  journal   = {Springer Monographs in Mathematics},
}

@Article{Woronowicz1976,
  author    = {Woronowicz, S. L.},
  journal   = {Reports on Mathematical Physics},
  title     = {Positive maps of low dimensional matrix algebras},
  year      = {1976},
  issn      = {0034-4877},
  number    = {2},
  pages     = {165--183},
  volume    = {10},
  date      = {1976},
  doi       = {10.1016/0034-4877(76)90038-0},
  publisher = {Elsevier {BV}},
}

@Article{Choi1980,
  author  = {Choi, M.-D.},
  journal = {J. Operat. Theor.},
  title   = {Some assorted inequalities for positive linear maps on {C}*-algebras},
  year    = {1980},
  number  = {2},
  pages   = {271--285},
  volume  = {4},
}

@Article{Labuschagne2006,
  author  = {Labuschagne, L. E. and Majewski, W. A. and Marciniak, M.},
  journal = {Expo. Math.},
  title   = {On $k$-decomposability of positive maps},
  year    = {2006},
  month   = jun,
  number  = {2},
  pages   = {103--125},
  volume  = {24},
  doi     = {10.1016/j.exmath.2005.07.002},
}

@Article{Chruscinski2022,
  author  = {Chru{\'{s}}ci{\'{n}}ski, D.},
  journal = {Phys. Rep.},
  title   = {Dynamical maps beyond Markovian regime},
  year    = {2022},
  month   = dec,
  pages   = {1--85},
  volume  = {992},
  doi     = {10.1016/j.physrep.2022.09.003},
}

@Article{Szczygielski2023,
  author  = {Szczygielski, K.},
  journal = {J. Phys. A: Math. Theor.},
  title   = {D-divisible quantum evolution families},
  year    = {2023},
  issn    = {1751-8121},
  number  = {48},
  pages   = {485202},
  volume  = {56},
  doi     = {10.1088/1751-8121/ad07c8},
}

@Book{Bengtsson2017,
  author    = {Bengtsson, I. and \.{Z}yczkowski, K.},
  publisher = {Cambridge University Pr.},
  title     = {Geometry of Quantum States},
  year      = {2017},
  isbn      = {1107026253},
  month     = sep,
  doi       = {10.1017/CBO9780511535048},
  ean       = {9781107026254},
  url       = {https://www.ebook.de/de/product/29273359/ingemar_bengtsson_karol_zyczkowski_geometry_of_quantum_states.html},
}

@Book{Alicki2006,
  author    = {Alicki, R. and Lendi, K.},
  publisher = {Springer},
  title     = {{Quantum {D}ynamical {S}emigroups and {A}pplications}},
  year      = {2006},
  address   = {Berlin Heidelberg},
  doi       = {10.1007/3-540-70861-8},
  pages     = {159-165},
}

@Book{Breuer2002,
  author    = {Breuer, H.-P. and Petruccione, F.},
  title     = {{The theory of open quantum systems}},
  year      = {2002},
  publisher = {Oxford University Press},
  doi       = {10.1093/acprof:oso/9780199213900.001.0001},
  address   = {New York},
}

@Book{Rivas2012,
  author    = {Rivas, {\'{A}}. and Huelga, S. F.},
  title     = {{Open {Q}uantum {S}ystems: {A}n {I}ntroduction}},
  year      = {2012},
  publisher = {Springer},
  doi       = {10.1007/978-3-642-23354-8},
  address   = {Berlin Heidelberg},
  issn      = {2191-5423},
}

@Article{Lindblad1976,
  author    = {Lindblad, G.},
  title     = {{On the generators of quantum dynamical semigroups}},
  journal   = {Commun. Math. Phys.},
  year      = {1976},
  volume    = {48},
  number    = {2},
  pages     = {119--130},
  issn      = {0010-3616},
  doi       = {10.1007/BF01608499},
  publisher = {Springer-Verlag},
}

@Article{Gorini1976,
  author  = {Gorini, V. and Kossakowski, A. and Sudarshan, E. C. G.},
  title   = {{Completely positive dynamical semigroups of {N}-level systems}},
  journal = {J. Math. Phys.},
  year    = {1976},
  volume  = {17},
  number  = {5},
  pages   = {821--825},
  issn    = {0022-2488},
  doi     = {10.1063/1.522979},
}

@Article{Holevo1993,
  author  = {Holevo, A. S.},
  journal = {Rep. Math. Phys.},
  title   = {A note on covariant dynamical semigroups},
  year    = {1993},
  issn    = {0034-4877},
  number  = {2},
  pages   = {211--216},
  volume  = {32},
  doi     = {10.1016/0034-4877(93)90014-6},
}

@Book{Rudin2013,
  author    = {Rudin, W.},
  publisher = {McGraw-Hill},
  title     = {Real and complex analysis},
  year      = {2013},
  address   = {New York, NY [u.a.]},
  edition   = {3. ed., internat. ed., [Nachdr.]},
  isbn      = {9780071002769},
  note      = {Literaturverz. S. 405 - 406},
  series    = {McGraw-Hill international editions},
  booktitle = {Real and complex analysis},
  pagetotal = {416},
}

@Article{Davies1974,
  author    = {Davies, E. B.},
  title     = {{Markovian master equations}},
  journal   = {Commun. Math. Phys.},
  year      = {1974},
  volume    = {39},
  number    = {2},
  pages     = {91--110},
  issn      = {0010-3616},
  doi       = {10.1007/BF01608389},
  publisher = {Springer-Verlag},
}

@Article{Davies1976,
  author  = {Davies, E. B.},
  journal = {Math. Ann.},
  title   = {{Markovian master equations. II}},
  year    = {1976},
  number  = {2},
  pages   = {147--158},
  volume  = {219},
  doi     = {10.1007/bf01351898},
}

@Article{Szczygielski2013,
  author    = {Szczygielski, K. and Gelbwaser-Klimovsky, D. and Alicki, R.},
  title     = {{Markovian master equation and thermodynamics of a two-level system in a strong laser field}},
  journal   = {Phys. Rev. E},
  year      = {2013},
  volume    = {87},
  number    = {012120},
  issue     = {1},
  pages     = {012120},
  issn      = {1539-3755},
  doi       = {10.1103/PhysRevE.87.012120},
  numpages  = {7},
  publisher = {American Physical Society},
}

@Article{Szczygielski2014,
  author  = {Szczygielski, K.},
  title   = {{On the application of {F}loquet theorem in development of time-dependent {L}indbladians}},
  doi     = {10.1063/1.4891401},
  issn    = {0022-2488},
  number  = {8},
  pages   = {083506},
  volume  = {55},
  journal = {J. Math. Phys.},
  year    = {2014},
}

@Article{Siudzinska2023,
  author    = {Siudzińska, K. and Studziński, M.},
  journal   = {Journal of Physics A: Mathematical and Theoretical},
  title     = {Adjusting phase-covariant qubit channel performance with non-unitality},
  year      = {2023},
  issn      = {1751-8121},
  month     = apr,
  number    = {20},
  pages     = {205301},
  volume    = {56},
  doi       = {10.1088/1751-8121/acccbf},
  publisher = {IOP Publishing},
}

@Article{Siudzinska2022,
  author    = {Siudzińska, K.},
  journal   = {J. Phys. A: Math. Theor.},
  title     = {Phase-covariant mixtures of non-unital qubit maps},
  year      = {2022},
  issn      = {1751-8121},
  number    = {40},
  pages     = {405303},
  volume    = {55},
  doi       = {10.1088/1751-8121/ac909b},
  publisher = {IOP Publishing},
}

@Article{Holevo1996,
  author    = {Holevo, A. S.},
  journal   = {J. Math. Phys.},
  title     = {Covariant quantum Markovian evolutions},
  year      = {1996},
  issn      = {1089-7658},
  number    = {4},
  pages     = {1812--1832},
  volume    = {37},
  doi       = {10.1063/1.531481},
  publisher = {AIP Publishing},
}

@Article{Siudzinska2025,
  author    = {Siudzińska, K. and Szczygielski, K.},
  journal   = {Phys. Lett. A},
  title     = {Decomposable dynamics on matrix algebras},
  year      = {2025},
  issn      = {0375-9601},
  month     = feb,
  pages     = {130185},
  volume    = {532},
  doi       = {10.1016/j.physleta.2024.130185},
  publisher = {Elsevier BV},
}

@Article{Kraus1971,
  author  = {Kraus, K.},
  journal = {Ann. Phys-new. York.},
  title   = {General state changes in quantum theory},
  year    = {1971},
  number  = {2},
  pages   = {311--335},
  volume  = {64},
  doi     = {10.1016/0003-4916(71)90108-4},
}

\end{document}